\documentclass[11pt,a4paper,reqno]{amsart}

\pdfoutput=1

\usepackage[utf8]{inputenc}
\usepackage[british]{babel}
\usepackage{amsmath,amssymb,amsthm}
\usepackage{thmtools,thm-restate}
\usepackage{hyperref,float,caption}
\hypersetup{
	colorlinks,
	linkcolor={red!60!black},
	citecolor={green!60!black},
	urlcolor={blue!60!black},
}
\usepackage[nameinlink]{cleveref}
\usepackage{geometry}
\usepackage{enumitem}
\usepackage[presets={vec-cev}]{letterswitharrows}
\usepackage[abbrev, msc-links, nobysame]{amsrefs}
\usepackage{tikz}
\usetikzlibrary{fadings}
\usetikzlibrary{arrows}
\usetikzlibrary {arrows.meta}

\usetikzlibrary{decorations.pathmorphing}

\tikzset{
	vertex/.style={circle, draw, minimum size=1.5em},
	edge/.style={->, > = latex'}
}

\usetikzlibrary{arrows}

\tikzset{
	dot/.style = {circle, draw = black, fill, minimum size=#1,
		inner sep=0pt, outer sep=0pt},
	dot/.default = 3.8pt
}
\usetikzlibrary{angles,quotes}

\usepackage{caption}
\usepackage{subcaption}

\usepackage{xcolor}
\definecolor{CornflowerBlue}{rgb}{0.39, 0.58, 0.93}
\definecolor{LavenderMagenta}{rgb}{0.93, 0.51, 0.93}
\definecolor{PastelOrange}{rgb}{1.0, 0.7, 0.28}

\newcommand{\x}[1]{\vecev{#1}}

\def\N{\mathbb N}

\def\R{\mathcal R}

\def\cN{\mathcal{N}}

\newcommand{\Down}[1]{\lceil #1 \rceil}

\newtheorem{theorem}{Theorem}[section]
\newtheorem{lemma}[theorem]{Lemma}

\newtheorem{proposition}[theorem]{Proposition}

\newtheorem{claim}{Claim}[theorem]
\theoremstyle{definition}

\newtheorem*{theorem*}{Theorem}

\crefname{enumi}{}{}
\crefname{enumii}{}{}
\crefformat{enumi}{#2#1#3}
\crefformat{enumii}{#2#1#3}
\Crefformat{enumi}{#2#1#3}
\Crefformat{enumii}{#2#1#3}

\crefname{definition}{definition}{definitions}
\crefformat{definition}{#2Definition~#1#3}
\Crefformat{definition}{#2Definition~#1#3}

\crefname{section}{section}{sections}
\crefformat{section}{#2Section~#1#3}
\Crefformat{section}{#2Section~#1#3}

\crefname{subsection}{Subsection}{subsections}
\crefformat{subsection}{#2Subsection~#1#3}
\Crefformat{subsection}{#2Subsection~#1#3}

\crefname{lemma}{lemma}{lemmata}
\crefformat{lemma}{#2Lemma~#1#3}
\Crefformat{lemma}{#2Lemma~#1#3}

\crefname{remark}{remark}{remarks}
\crefformat{remark}{#2Remark~#1#3}
\Crefformat{remark}{#2Remark~#1#3}

\crefname{theorem}{theorem}{theorems}
\crefformat{theorem}{#2Theorem~#1#3}
\Crefformat{theorem}{#2Theorem~#1#3}

\crefname{corollary}{corollary}{corollaries}
\crefformat{corollary}{#2Corollary~#1#3}
\Crefformat{corollary}{#2Corollary~#1#3}

\crefname{figure}{figure}{figures}
\crefformat{figure}{#2Figure~#1#3}
\Crefformat{figure}{#2Figure~#1#3}

\crefname{proposition}{proposition}{propositions}
\crefformat{proposition}{#2Proposition~#1#3}
\Crefformat{proposition}{#2Proposition~#1#3}

\crefname{observation}{observation}{observations}
\crefformat{observation}{#2Observation~#1#3}
\Crefformat{observation}{#2Observation~#1#3}

\def\lqedsymbol{\ifmmode$\lrcorner$\else{\unskip\nobreak\hfil
		\penalty50\hskip1em\null\nobreak\hfil$\rule{1.2ex}{1.2ex}$
		\parfillskip=0pt\finalhyphendemerits=0\endgraf}\fi}

\newenvironment{claimproof}[1][\proofname]
{%
	\proof[#1]%
}
{%
	\endproof%
}

\title{Halin's grid theorem for digraphs}

\author{Florian Reich}
\address{Universit{\"a}t Hamburg, Department of Mathematics, Bundesstra{\ss}e~55 (Geomatikum), 20146~Hamburg, Germany}
\email{florian.reich@uni-hamburg.de}

\begin{document}
\begin{abstract}
	Halin showed that every thick end of every graph contains an infinite grid.
	We extend Halin's theorem to digraphs.
	More precisely, we show that for every infinite family $\mathcal{R}$ of disjoint equivalent out-rays there is a grid whose vertical rays are contained in $\mathcal{R}$.
	Furthermore, we obtain similar results for in-rays and necklaces.
\end{abstract}
	
\maketitle

\section{Introduction}
Two rays of a graph $G$ are \emph{equivalent} if there exist infinitely many disjoint paths between them in $G$.
The \emph{ends} of $G$ are the equivalence classes of rays in $G$ and we call an end \emph{thick} if it contains infinitely many disjoint rays.
Halin~\cite{halin1965maximalzahl} proved that for every thick end $\omega$ in $G$ there exists a subdivided hexagonal half-grid in $G$ witnessing that $\omega$ is thick.
Kurkofka, Melcher and Pitz~\cite{kurkofka2022strengthening} strengthened Halin's seminal result in the following way:

\begin{theorem}\label{thm:undirected}
	For every infinite family $\mathcal{R}$ of disjoint equivalent rays in a graph $G$, there exists a subdivision of a hexagonal half-grid $G'$ such that each vertical ray of $G'$ is an element of $\mathcal{R}$.
\end{theorem}
\noindent
For certain classes of graphs thick ends are even witnessed by subdivided full hexagonal grids,
which has been studied by
Heuer~\cite{heuer2017excluding}, and by Hamann and Georgakopoulos~\cite{georgakopoulos2024full}.

In this paper we investigate grids in digraphs that are necessary and sufficient for thick ends of digraphs.
There exist two common notions of directed ends.
The first notion is due to Zuther~\cite{zuther1998ends} and defines \emph{ends} of digraphs as equivalence classes consisting of in- and out-rays.
An \emph{in-ray} is a digraph whose underlying undirected graph is a ray and whose edges are oriented towards the unique vertex of degree $1$.
Similarly, \emph{out-rays} are defined.
Given an in- or out-ray $R_1$ and an in- or out-ray $R_2$, we say $R_1$ and $R_2$ are \emph{equivalent} if there exist infinitely many disjoint directed $R_1$--$R_2$~paths and infinitely many disjoint directed $R_2$--$R_1$~paths in $D$~\cite{zuther1998ends}.
An end is \emph{thick} if it contains infinitely many disjoint out-rays.

We determine two unavoidable grid-like structures in these thick ends:
On the one hand, the \emph{bidirected quarter-grid}, where two consecutive out-rays are connected in both directions (see~\cref{fig:bidirected_quarter_grid}).
On the other hand, the \emph{cyclic quarter-grids}, where two consecutive out-rays are connected in one direction and in the other direction the out-rays are connected to the first out-ray (see~\Cref{fig:dominated_directed_quarter_grid_in,fig:dominated_directed_quarter_grid_out}).
In all these grid the mentioned out-rays, called \emph{vertical} out-rays, are disjoint and equivalent, and thus witness that the corresponding end is thick.
Moreover, we will show that all rays contained in such a grid are in the same (thick) end.
Thick ends for families of in-rays and the two grid-like structures for in-rays, called \emph{reversed bidirected quarter-grid} and \emph{reversed cyclic quarter-grid}, are defined similarly.

Our main result extends~\cref{thm:undirected} to Zuther ends:

	\begin{figure}
		\centering
		\begin{subfigure}[b]{0.3\textwidth}
		\centering
		\begin{tikzpicture}

				
				\foreach \y/\z in {0/1,1/1.5,1.5/2.5,2.5/4,4/5,5/7.5,7.5/8.5,8.5/12,12/13}{
					\filldraw ({0*0.8},{0.5*\y}) circle (1pt);
					\draw[edge] ({0*0.8},{0.5*\y+0.05}) to ({0*0.8},{0.5*(\z)-0.05});
				}
				
				\foreach \y/\z in {0/1,1/1.5,1.5/2.5,2.5/3,3/3.5,3.5/4,4/5,5/5.5,5.5/7,7/7.5,7.5/8.5,8.5/9,9/11.5,11.5/12,12/13}{
					\filldraw ({1*0.8},{0.5*\y}) circle (1pt);
					\draw[edge] ({1*0.8},{0.5*\y+0.05}) to ({1*0.8},{0.5*(\z)-0.05});
				}
				
				\foreach \y/\z in {2/3,3/3.5,3.5/5.5,5.5/6,6/6.5,6.5/7,7/9,9/9.5,9.5/11,11/11.5,11.5/13}{
					\filldraw ({2*0.8},{0.5*\y}) circle (1pt);
					\draw[edge] ({2*0.8},{0.5*\y+0.05}) to ({2*0.8},{0.5*(\z)-0.05});
				}
				
				\foreach \y/\z in {5/6,6/6.5,6.5/9.5,9.5/10,10/10.5,10.5/11,11/13}{
					\filldraw ({3*0.8},{0.5*\y}) circle (1pt);
					\draw[edge] ({3*0.8},{0.5*\y+0.05}) to ({3*0.8},{0.5*(\z)-0.05});
				}
				
				\foreach \y/\z in {9/10,10/10.5,10.5/13}{
					\filldraw ({4*0.8},{0.5*\y}) circle (1pt);
					\draw[edge] ({4*0.8},{0.5*\y+0.05}) to ({4*0.8},{0.5*(\z)-0.05});
				}
				
				
				\foreach \x/\y in {0/1,0/2.5,1/3,0/5,1/5.5,2/6,0/8.5,1/9,2/9.5,3/10}{
					\draw[edge] ({\x*0.8+0.05},{0.5*\y}) to ({\x*0.8+0.95*0.8}, {0.5*\y});
				}
				
				\foreach \x/\y in {1/1.5,2/3.5,1/4,3/6.5,2/7,1/7.5,4/10.5,3/11,2/11.5,1/12}{
					\draw[edge] ({\x*0.8-0.05},{0.5*\y}) to ({\x*0.8-0.95*0.8}, {0.5*\y});
				}
				
				
				\foreach \x in {0,1,2,3,4}{
					\filldraw[gray] ({\x*0.8},6.6) circle (0.5pt);
					\filldraw[gray] ({\x*0.8},6.7) circle (0.5pt);
					\filldraw[gray] ({\x*0.8},6.8) circle (0.5pt);
				}
				
				
				\filldraw[gray] ({4.7*0.8},5.5) circle (0.5pt);
				\filldraw[gray] ({4.9*0.8},5.5) circle (0.5pt);
				\filldraw[gray] ({5.1*0.8},5.5) circle (0.5pt);

				\draw (0,-0.5) node {$R_1$};
				\draw (0.8,-0.5) node {$R_2$};
				\draw (1.6,0.5) node {$R_3$};
				\draw (2.4,2) node {$R_4$};
				\draw (3.2,4) node {$R_5$};

		\end{tikzpicture}
		\caption{The bidirected quarter-grid.}
		\label{fig:bidirected_quarter_grid}
		\end{subfigure}
		\begin{subfigure}[b]{0.31\textwidth}
		\centering
		\begin{tikzpicture}
			
			\foreach \y/\z in {0/1,1/2,2/3,3/4.5,4.5/5.5,5.5/7.5,7.5/8.5,8.5/11,11/13}{
				\filldraw ({0*0.8},{0.5*\y}) circle (1pt);
				\draw[edge] ({0*0.8},{0.5*\y+0.05}) to ({0*0.8},{0.5*(\z)-0.05});
			}
			
			\foreach \y/\z in {0/1,1/2,2/4,4/4.5,4.5/7,7/7.5,7.5/10.5,10.5/11,11/13}{
				\filldraw ({1*0.8},{0.5*\y}) circle (1pt);
				\draw[edge] ({1*0.8},{0.5*\y+0.05}) to ({1*0.8},{0.5*(\z)-0.05});
			}
			
			\foreach \y/\z in {2/3,3/4,4/6.5,6.5/7,7/10,10/10.5,10.5/13}{
				\filldraw ({2*0.8},{0.5*\y}) circle (1pt);
				\draw[edge] ({2*0.8},{0.5*\y+0.05}) to ({2*0.8},{0.5*(\z)-0.05});
			}
			
			\foreach \y/\z in {4.5/5.5,5.5/6.5,6.5/9.5,9.5/10,10/13}{
				\filldraw ({3*0.8},{0.5*\y}) circle (1pt);
				\draw[edge] ({3*0.8},{0.5*\y+0.05}) to ({3*0.8},{0.5*(\z)-0.05});
			}
			
			\foreach \y/\z in {7.5/8.5,8.5/9.5,9.5/13}{
				\filldraw ({4*0.8},{0.5*\y}) circle (1pt);
				\draw[edge] ({4*0.8},{0.5*\y+0.05}) to ({4*0.8},{0.5*(\z)-0.05});
			}
			
			
			\foreach \x/\y in {0/2,0/4.5,1/4,0/7.5,1/7,2/6.5,0/11,1/10.5,2/10,3/9.5}{
				\draw[edge] ({\x*0.8+0.95*0.8}, 0.5*\y) to ({\x*0.8+0.05},0.5*\y);
			}
			
			\foreach \x/\y in {1/1,2/3,3/5.5,4/8.5}{
				\draw[white, bend left=20, line width=0.2cm] ({0.05*0.8}, 0.5*\y) to ({\x*0.8-0.05},0.5*\y);
				\draw[edge, bend left=20, shorten <=1pt, shorten >=1pt] ({0*0.8}, 0.5*\y) to ({\x*0.8},0.5*\y);
			}
			
			
			\foreach \x in {0,1,2,3,4}{
				\filldraw[gray] ({\x*0.8},6.6) circle (0.5pt);
				\filldraw[gray] ({\x*0.8},6.7) circle (0.5pt);
				\filldraw[gray] ({\x*0.8},6.8) circle (0.5pt);
			}
			
			
			\filldraw[gray] ({4.7*0.8},5.5) circle (0.5pt);
			\filldraw[gray] ({4.9*0.8},5.5) circle (0.5pt);
			\filldraw[gray] ({5.1*0.8},5.5) circle (0.5pt);
			
			\draw (0,-0.5) node {$R_1$};
			\draw (0.8,-0.5) node {$R_2$};
			\draw (1.6,0.5) node {$R_3$};
			\draw (2.4,1.75) node {$R_4$};
			\draw (3.2,3.25) node {$R_5$};
			
		\end{tikzpicture}
		\caption{The descending cyclic quarter-grid.}
		\label{fig:dominated_directed_quarter_grid_in}
	\end{subfigure}
	\begin{subfigure}[b]{0.31\textwidth}
	\centering
	\begin{tikzpicture}

\foreach \y/\z in {0/1,1/1.5,1.5/3,3/4,4/5.5,5.5/7,7/8.5,8.5/10.5,10.5/13}{
	\filldraw ({8*0.8},{0.5*\y}) circle (1pt);
	\draw[edge] ({8*0.8},{0.5*\y+0.05}) to ({8*0.8},{0.5*(\z)-0.05});
}

\foreach \y/\z in {0/1,1/1.5,1.5/3,3/3.5,3.5/5.5,5.5/6,6/8.5,8.5/9,9/13}{
	\filldraw ({9*0.8},{0.5*\y}) circle (1pt);
	\draw[edge] ({9*0.8},{0.5*\y+0.05}) to ({9*0.8},{0.5*(\z)-0.05});
}

\foreach \y/\z in {2.5/3.5,3.5/4,4/6,6/6.5,6.5/9,9/9.5,9.5/13}{
	\filldraw ({10*0.8},{0.5*\y}) circle (1pt);
	\draw[edge] ({10*0.8},{0.5*\y+0.05}) to ({10*0.8},{0.5*(\z)-0.05});
}

\foreach \y/\z in {5.5/6.5,6.5/7,7/9.5,9.5/10,10/13}{
	\filldraw ({11*0.8},{0.5*\y}) circle (1pt);
	\draw[edge] ({11*0.8},{0.5*\y+0.05}) to ({11*0.8},{0.5*(\z)-0.05});
}

\foreach \y/\z in {9/10,10/10.5,10.5/13}{
	\filldraw ({12*0.8},{0.5*\y}) circle (1pt);
	\draw[edge] ({12*0.8},{0.5*\y+0.05}) to ({12*0.8},{0.5*(\z)-0.05});
}


\foreach \x/\y in {8/1,8/3,9/3.5,8/5.5,9/6,10/6.5,8/8.5,9/9,10/9.5,11/10}{
	\draw[edge] ({\x*0.8+0.05}, 0.5*\y) to ({\x*0.8+0.95*0.8}, 0.5*\y);
}

\foreach \x/\y in {9/1.5,10/4,11/7,12/10.5}{
	\draw[white, bend right=20, line width=0.2cm, shorten <=5pt, shorten >=5pt] ({\x*0.8},0.5*\y) to ({8*0.8}, 0.5*\y);
	\draw[edge, bend right=20, shorten <=1pt, shorten >=1pt] ({\x*0.8},0.5*\y) to ({8*0.8}, 0.5*\y);
}


			\foreach \x in {8,9,10,11,12}{
	\filldraw[gray] ({\x*0.8},6.6) circle (0.5pt);
	\filldraw[gray] ({\x*0.8},6.7) circle (0.5pt);
	\filldraw[gray] ({\x*0.8},6.8) circle (0.5pt);
}


\filldraw[gray] ({12.7*0.8},5.5) circle (0.5pt);
\filldraw[gray] ({12.9*0.8},5.5) circle (0.5pt);
\filldraw[gray] ({13.1*0.8},5.5) circle (0.5pt);

				\draw (6.4,-0.5) node {$R_1$};
				\draw (7.2,-0.5) node {$R_2$};
				\draw (8,0.75) node {$R_3$};
				\draw (8.8,2.25) node {$R_4$};
				\draw (9.6,4) node {$R_5$};
		
	\end{tikzpicture}
	\caption{The ascending cyclic quarter-grid.}
	\label{fig:dominated_directed_quarter_grid_out}
\end{subfigure}
	\caption{The necessary and sufficient grids for thick Zuther ends. The vertical rays of the respective grids are $R_1, R_2, \dots$.}
	\end{figure}

\begin{restatable}{theorem}{thmMain}\label{thm:main}
	For every infinite family $(R_i)_{i \in I}$ of disjoint equivalent out-rays (in-rays) in a digraph $D$ there exists a subdivision $D'$ of either a (reversed) bidirected quarter-grid or a (reversed) cyclic quarter-grid in $D$ such that each vertical out-ray (in-ray) of $D'$ is an element of $(R_i)_{i \in I}$.
\end{restatable}
\noindent
We remark that both types of quarter-grids are indeed necessary for a statement like~\cref{thm:main} since the bidirected quarter-grid does not contain a subdivision of a cyclic quarter-grid whose vertical rays are contained in $(R_i)_{i \in I}$ and vice versa.

By relaxing the constraints so that the vertical rays only have to be equivalent to $(R_i)_{i \in I}$ the cyclic quarter-grid can be omitted:

\begin{restatable}{theorem}{thmZutherEnd}\label{thm:zuther_end}
	For every infinite family $(R_i)_{i \in I}$ of disjoint equivalent out-rays (in-rays) in a digraph~$D$ there exists a subdivision of the (reversed) bidirected quarter-grid $D'$ such that each vertical out-ray (in-ray) of $D'$ is equivalent to the rays in $(R_i)_{i \in I}$.
\end{restatable}
\noindent
We deduce \cref{thm:zuther_end} from~\cref{thm:main}.
Independently, Hamann and Heuer~\cite{hamann} recently proved \cref{thm:zuther_end}.\footnote{Hamann and Heuer's notion of (reversed) bidirected quarter-grid is slightly different, but their (reversed) bidirected quarter-grid is contained in a subdivision of our (reversed) bidirected quarter-grid and vice versa.}
Furthermore, Zuther showed a special case of~\cref{thm:zuther_end} for strictly increasing sequences of ends that contain out-rays \cite{zuther1998ends}*{Theorem 3.1}.

The second notion of ends in digraphs is due to B\"urger and Melcher, who define \emph{ends} of digraphs as equivalence classes of necklaces~\cite{burger2020ends}.
A \emph{necklace} is a strongly connected subgraph $N$ for which there exists a family $(H_n)_{n \in \N}$ of finite strongly connected subgraphs such that $N = \bigcup_{n \in \N} H_n$ and $H_i \cap H_j \neq \emptyset$ holds if and only if $|i - j| \leq 1$ for every $i, j \in \N$~\cite{bowler2024connectoids}.\footnote{B\"urger and Melcher used an equivalent definition of necklaces in~\cite{burger2020ends}.}
Two necklaces $N_1$ and $N_2$ are \emph{equivalent} if there exist infinitely many disjoint directed $N_1$--$N_2$~paths and infinitely many disjoint directed $N_2$--$N_1$~paths, and an end is \emph{thick} if it contains infinitely many disjoint necklaces~\cite{burger2020ends}.
This notion of ends reflects canonically the strong connectivity of its host digraph and in particular satisfies the direction theorem for digraphs~\cite{burger2020ends}.

The grid-like structures witnessing this kind of thick ends are similar to the bidirected and the cyclic quarter-grid:
\emph{Bidirected necklace grids} and \emph{cyclic quarter-grids} are obtained from subdivisions of the corresponding quarter-grids by replacing the vertical rays by necklaces.
We prove:

\begin{restatable}{theorem}{thmNecklace}\label{thm:necklace}
	For every infinite family $\cN$ of disjoint equivalent necklaces in a digraph $D$ there exists either a bidirected necklace grid $D'$ or a cyclic necklace grid $D'$ in $D$ such that each vertical necklace of $D'$ is an element of $\cN$.
\end{restatable}

We will observe in \cref{sec:necklace_grids} that there exists a cyclic necklace grid that does not contain a bidirected necklace grid and vice versa.
Thus both types of grids are indeed necessary for~\cref{thm:necklace}, and this holds true even if we relax the constraints on the vertical necklaces as in~\cref{thm:zuther_end}.
We will prove~\cref{thm:necklace} by constructing an auxiliary graph to which we can apply~\cref{thm:main}.

In the proof of our main result, \cref{thm:main}, we analyse the connectivity between distinct out-rays of $(R_i)_{i \in I}$, for which we follow the idea of Kurkofka, Melcher and Pitz~\cite{kurkofka2022strengthening} to consider the auxiliary graph obtained by contracting the rays $(R_i)_{i \in I}$:
we fix an out-ray $S$ that intersects every element of $(R_i)_{i \in I}$ infinitely often.
For every $J \subseteq I$ let $D(J)$\footnote{The digraph $D(J)$ depends on the choice of $S$, but this choice will always be clear from the context.} be the digraph obtained from $S \cup \bigcup_{i \in J} R_i$ by firstly contracting the out-ray $R_i$ to a single vertex~$i$ for every $i \in J$, secondly suppressing vertices of in- and out-degree $1$ and thirdly removing loops.
By construction, there is a canonical one-to-one~correspondence $\sigma$ between the edges of $D(J)$ and the directed paths in $S$ that have endpoints in distinct out-rays $(R_i)_{i \in J}$ and are internally disjoint to $\bigcup_{i \in J} R_i$, i.e.\ $\sigma(uv)$ is a directed path starting in $R_u$ and ending in $R_v$ that is internally disjoint to $\bigcup_{i \in J} R_i$.

Further, we consider the auxiliary graph $D^\infty(J)$ of $D(J)$ that captures edges of infinite multiplicity, i.e\ the digraph obtained from $D(J)$ by removing edges of finite multiplicity and replacing edges of infinite multiplicity by single edges.

This paper is organised as follows.
In~\cref{sec:preliminiaries} we introduce basic notations and show that there exists an out-ray $S$ that intersects every element of $(R_i)_{i \in I}$ infinitely often.
We show in~\cref{sec:component} that~\cref{thm:main} holds true if $D^\infty(J)$ contains an infinite strong component and in~\cref{sec:rays} that~\cref{thm:main} holds true if $D^\infty(J)$ contains either an in-ray or an out-ray.
We analyse the structure of $D$ if there does not exist $J \subseteq I$ such that $D^\infty(J)$ contains an infinite strong component, an in-ray or an out-ray in~\cref{sec:subfamilies}.
In~\cref{sec:proof_main_theorem} we combine the results of~\Cref{sec:component,sec:rays,sec:subfamilies} to prove~\cref{thm:main}, and further deduce~\cref{thm:zuther_end}.
Finally, we show~\cref{thm:necklace} in~\cref{sec:necklace_grids} and argue that both types of necklace grids are indeed necessary.

\section{Preliminaries}\label{sec:preliminiaries}

For standard notations we refer to Diestel's book~\cite{diestel}.
For simplicity, we refer to directed paths as \emph{paths}.
Given sets $A,B$ of vertices of a digraph $D$, an \emph{$A$--$B$~path} is a path in $D$ that starts in $A$, ends in $B$ and is internally disjoint to $A \cup B$.
Similarly, we define \emph{$A$--$B$~paths} for subgraphs $A,B$, and \emph{$a$--$b$~paths} for vertices $a$ and $b$.
Furthermore, an \emph{$A$-path} is an $A$--$A$~path.

A digraph $D$ is \emph{strongly connected} if for every two vertices $v,w \in V(D)$ there exists a $v$--$w$~path in $D$.
The maximal strongly connected subgraphs of a digraph $D$ are called the \emph{strong components} of $D$.
Furthermore, the components of the underlying undirected graph of $D$ are called the \emph{weak components} of $D$.
Given a digraph $D$, let $\Dv$ be the digraph obtained from $D$ by reversing the orientation of each edge.
Given a graph $G$, let \emph{$\mathcal{D}(G)$} be the digraph obtained from $G$ by replacing each edge $uv \in E(G)$ by directed edges $(u,v), (v,u)$.

The unique vertex of an out- or in-ray $R$ that is incident with precisely one edge is the \emph{root} of $R$.
Given an out-ray or an in-ray $R$, we write $vR$ for the unique out-/in-ray in $R$ with root $v \in V(R)$.
For $v,w \in V(R)$ let $v \leq_R w$ if and only if $w \in V(vR)$.
Additionally, we let $vRw$ be the path in $R$ that has endpoints $v$ and $w$.
For every subset $J \subseteq I$, every set $X \subseteq V(D)$ and every family $({R}_i)_{i \in I}$ of disjoint equivalent in-rays or of disjoint equivalent out-rays in a digraph $D$, we define $\Down{X}_J:= X \cup \bigcup_{j \in J} \{v \in V(R_j): V(v R_j) \cap X \neq \emptyset \}$.
An \emph{in-arborescence} is a digraph whose underlying undirected graph is a rooted tree such that all edges are oriented towards the root. Similarly, \emph{out-arborescences} are defined.

\begin{proposition}\label{prop:existence_out_ray}
	Let $D$ be a digraph and let $(R_i)_{i \in I}$ be a countable family of disjoint equivalent out-rays (in-rays) in $D$.
	There exists an out-ray (in-ray) $S$ that intersects every element of $(R_i)_{i \in I}$ infinitely often.
\end{proposition}

\begin{proof}
	We begin with the case that $(R_i)_{i \in I}$ is a family of out-rays.
	Let $(\alpha(n))_{n \in \N}$ be a sequence of elements in $I$ such that each element of $I$ occurs infinitely often.
	We construct recursively a strictly increasing sequence of nested paths $(P_n)_{n \in \N}$ starting in the same vertex such that $P_{n}$ ends in $x_n \in V(R_{\alpha(n)})$ with $x_n R_{\alpha(n)} \cap P_n = \{x_n\}$ for every $n \in \N$.
	
	Let $x_1$ be some vertex of $R_{\alpha(1)}$ and set $P_1 := \{x_1\}$.
	If $P_n$ has been constructed for some $n \in \N$, then let $Q_{n+1}$ be some $R_{\alpha(n)}$--$R_{\alpha(n+1)}$~path avoiding $\Down{V(P_n)}_I$, which exists since $R_{\alpha(n)}$ and $R_{\alpha(n+1)}$ are equivalent.
	Further, let $O_{n+1}$ be the path in $R_{\alpha(n)}$ starting in $x_n$ and ending in the start vertex of $Q_{n+1}$.
	Then the concatenation $P_{n+1}$ of $P_n, O_{n+1}$ and $Q_{n+1}$ is as desired.
	This completes the construction of $(P_n)_{n \in \N}$.
	The out-ray $\bigcup_{n \in \N} P_n$ is as desired.
	
	For a family $(R_i)_{i \in I}$ of in-rays, we obtain the desired in-ray by applying the result for out-rays to $\Dv$.
\end{proof}
For families of in-rays we define $D(J)$ and $D^\infty(J)$ in a similar way as for out-rays by considering an in-ray $S$ intersecting every in-ray in $J$ infinitely often.
Given a family $(R_i)_{i \in I}$ of disjoint equivalent in-rays or disjoint equivalent out-rays, let $\sigma$ be the canonical one-to-one~correspondence between the edges of $D(J)$ and the $(\bigcup_{i \in J} R_i)$-paths in $S$.
Note that $\sigma(e)$ and $\sigma(f)$ are internally disjoint and they intersect only if the head of $e$ is the tail of $f$ or vice versa for every $e \neq f \in E(D(J))$.
Furthermore, each edge of $D(J)$ with endpoints in $J'$ corresponds to an edge of $D(J')$ for every $J' \subseteq J$.

\begin{lemma}\label{lem:out-degree}
	Let $D$ be a digraph and let $(R_i)_{i \in I}$ be an infinite family of disjoint equivalent out-rays (in-rays) in $D$.
	For every finite subset $J \subseteq I$ there is either
	\begin{itemize}
		\item an element of $J$ with infinite out-degree (in-degree) in $D(I)$, or
		\item an edge in $D^\infty(I)$ with tail (head) in $J$ and head (tail) in $I \setminus J$.
	\end{itemize}
\end{lemma}
\begin{proof}
	Since the rays in $(R_i)_{i \in I}$ are equivalent, there are infinitely many disjoint $(\bigcup_{i \in I}R_i)$-paths that start in $(R_j)_{j \in J}$ and end in $(R_i)_{i \in I \setminus J}$.
	Thus there are infinitely many edges of $D(I)$ starting in $J$ and ending in $I \setminus J$.
	As $J$ is finite, there exists $j \in J$ such that infinitely many edges in $D(I)$ start in $j$ and end in $I \setminus J$.
	Either there are infinitely many vertices in $I \setminus J$ in which at least one of these edges end or there exists a vertex in $I \setminus J$ in which infinitely many of these edges end.
	In the former case $j$ has infinite out-degree in $D(I)$ and in the latter case there is an edge in $D^\infty(I)$ with tail $j$ and head in $I \setminus J$.
	
	By considering $\Dv$ we can deduce the corresponding statement for ingoing edges.
\end{proof}

For a sequence $\mathcal{A}:= (i_1, \dots, i_n)$ of elements in $I$ with $n \geq 2$, we call a path $P$ an \emph{$\mathcal{A}$--staircase} if 
$P$ is the concatenation of paths $P_1, Q_2, P_2, Q_3, \dots, Q_{n-1}, P_{n-1}$, where
\begin{enumerate}[label=(\alph*)]
	\item\label{itm:staircase_1} $P_j$ is an $R_{i_j}$--$R_{i_{j+1}}$~path that is internally disjoint to $\bigcup_{i \in I} R_i$ for every $j \in [n-1]$, and
	\item\label{itm:staircase_2} $Q_j$ is a non-trivial path in $R_{i_j}$ that starts in the end vertex of $P_{j-1}$ and ends in the start vertex of $P_j$ for every $j \in [n-1] \setminus \{1\}$.
\end{enumerate}
If each of these paths consists of precisely one edge, we call $P$ \emph{simple}.

\begin{proposition}\label{prop:zickzack}
	Let $D$ be a digraph and let $(R_i)_{i \in I}$ be a family of disjoint equivalent out-rays or a family of disjoint equivalent in-rays.
	For every finite set $X \subseteq V(D)$ and every sequence $\mathcal{A}:= (i_1, \dots, i_n)$ of distinct elements of $I$ with $i_j i_{j+1} \in E(D^\infty(I))$ there exists an $\mathcal{A}$--staircase that is disjoint to $X$.
\end{proposition}

\begin{proof}
	We prove the statement by induction on $n$.
	For $n = 2$, some $R_{i_1}$--$R_{i_2}$~path that is internally disjoint to $\bigcup_{i \in I} R_i$ and avoids $X$, which exists since $i_1 i_2 \in E(D^\infty(I))$, is the desired $\mathcal{A}$--staircase avoiding $X$.
	
	If $(R_i)_{i \in I}$ is a family of out-rays, we assume that $\hat{P}$ is some $\hat{\mathcal{A}}$--staircase avoiding $\Down{X}_I$, where $\hat{\mathcal{A}} := i_1, \dots, i_{n-1}$ for some $n > 2$.
	Let $P_{n-1}$ be some $R_{i_{n-1}}$--$R_{i_n}$~path avoiding $\Down{X}_I \cup \Down{\hat{P}}_I$ that is internally disjoint to $\bigcup_{i \in I} R_i$, which exists since $i_{n-1} i_n \in E(D^\infty(I))$.
	Further let $Q_{n-1}$ be the subpath of $R_{i_{n-1}}$ starting in the end vertex of $\hat{P}$ and ending in the start vertex of $P_{n-1}$.
	Note that $Q_{n-1}$ avoids $X$ since $\hat{P}$ avoids $\Down{X}_I$.
	Then the concatenation of $P_{n-1}, Q_{n-1}$ and $\hat{P}$ is the desired $\mathcal{A}$--staircase avoiding $X$.
	
	If $(R_i)_{i \in I}$ is a family of in-rays, we consider $\hat{\mathcal{A}} := i_2, \dots, i_{n}$ for some $n > 2$.
	By a similar construction we obtain the desired $\mathcal{A}$--staircase avoiding $X$.
\end{proof}

	Finally, we define the bidirected and cyclic quarter-grids formally.
	The bidirected and the cyclic quarter-grid $D'$ are constructed from the disjoint union of out-rays $(R_i)_{i \in \N}$ by adding digraphs $(G_n)_{n \geq 2}$, called \emph{girders}, such that
	\begin{itemize}
		\item $G_{n}$ avoids $\bigcup_{i \in [n-1] \setminus \{1\}}\Down{G_i}_\N$,
		\item $\bigcup_{n \geq 2} G_n \cup \bigcup_{i \in \N} R_i$ does not contain a vertex of in-degree and out-degree one, and
		\item the first vertex of $R_i$ is disjoint from $\bigcup_{n \geq 2} G_n$ for every $i \in \N$.
	\end{itemize}
	If $G_n$ is the concatenation of a simple $1, \dots, n$--staircase $S_n'$, an edge in $R_n$ and a simple $n, \dots, 1$--staircase $S_n''$ such that $S_n''$ avoids $\Down{S_n'}_\N$ for every $n \geq 2$, then $D'$ is a \emph{bidirected quarter-grid}.
	
	If $G_n$ is the concatenation of a simple $1, \dots, n$--staircase $S_n$, an edge in $R_n$ and an $R_n$--$R_1$~edge $A_n$ such that $A_n$ avoids $\Down{S_n}_\N$ for every $n \geq 2$, then $D'$ is an \emph{ascending cyclic quarter-grid}.
	Similarly, if $G_n$ is the concatenation of an $R_1$--$R_n$~edge $A_n$, an edge in $R_n$ and a simple $n, \dots, 1$--staircase $S_n$ such that $S_n$ avoids $\Down{A_n}_\N$ for every $n \geq 2$, then $D'$ is a \emph{descending cyclic quarter-grid}.\footnote{Hamann and Heuer defined the ascending cyclic quarter-grid and the descending cyclic quarter-grid slightly differently~\cite{hamann}. By deleting the first vertex of each $R_i$ we obtain their cyclic quarter-grids.}
	In both cases, we refer to $A_n$ and the subdivisions of $A_n$ as \emph{arches}.
	
	\begin{proposition}
		Let $D'$ be a bidirected quarter grid or a cyclic quarter-grid.
		Then all out-rays contained in $D'$ are equivalent.
		Furthermore, there is no in-ray in $D'$.
	\end{proposition}
	\begin{proof}
		Let $(R_n)_{n \in \N}$ be the vertical rays of $D'$ and let $(G_n)_{n \in \N}$ be the girders of $D'$.
		We set $Z_n:= \Down{G_n}_{\N} \cup \bigcup_{i \in [n]} V(G_i)$ and note that $Z_1 \subseteq Z_2 \subseteq \dots$ is an increasing sequence of finite subsets of $V(D')$ such that $\bigcup_{n \in \N} Z_n = V(D')$.
		There does not exists an $n \in \N$ for which there is an edge with head in $Z_n$ and tail in $V(D') \setminus Z_n$, which implies that there is no in-ray in $D'$.
		Let $R$ be an arbitrary out-ray in $D'$.
		It suffices to show that $R$ is equivalent to $R_1$.
		Suppose for a contradiction that there is a finite set $X \subseteq V(D')$ such that there is either no $R_1$--$R$~path in $D' -X$ or no $R$--$R_1$~path in $D' - X$.
		
		Since $X$ is finite, there is $n \in \N$ such that $X \subseteq Z_n$.
		Then there is $z \in V(R) \setminus Z_{n+1}$.
		We can assume without loss of generality that $z$ is no the start vertex of $R$ and thus not the start vertex of some $R_i$.
		If $z$ is contained in some girder $G_k$, then $k \geq n+1$, which implies that $G_k \subseteq D' - X$ contains an $R_1$--$R$~path and an $R$--$R_1$~path.
		Otherwise, $z \in V(R_i)$ for some $i \in \N$, and thus there are $k, \ell \in \N$ with $n+1 \leq k \leq \ell$ such that the girders $G_k$ and $G_\ell$ witness that there exist an $R_1$--$R$~path and an $R$--$R_1$~path in $D' - X$, a contradiction.
	\end{proof}

\section{Infinite strong component}\label{sec:component}
In this section we show that \cref{thm:main} holds true if there exists $J \subseteq I$ such that $D^\infty(J)$ contains an infinite strong component.

\begin{lemma}\label{lem:thick_component}
	Let $(R_i)_{i \in I}$ be an infinite family of disjoint equivalent out-rays in a digraph~$D$.
	If $D^\infty(I)$ has an infinite strong component, then $D$ contains a subdivision $D'$ of either a bidirected quarter-grid or a cyclic quarter-grid such that each vertical out-ray of~$D'$ is an element of $(R_i)_{i \in I}$.
\end{lemma}

The proof of~\cref{lem:thick_component} builds on a result \cite{infinite}*{Corollary 1.4} about strongly connected butterfly minors in strongly connected digraphs.
We begin by introducing butterfly minors.
Given digraphs $D$ and $H$, we call a map $\mu $ assigning every $e\in E(H)$ an edge of $D$ and every $v \in V(H)$ a subgraph of $D$, such that
\begin{itemize}
	\item $\mu (v) \cap \mu (w) = \emptyset$ for every $v \neq w \in V(H)$,
	\item for every $v \in V(H)$ the digraph $\mu(v)$ is a union of an in-arborescence $T_{\mathrm{in}}^{\mu(v)}$ and an out-arborescence $T_{\mathrm{out}}^{\mu(v)}$ that only have their roots in common, and
	\item for every $e=(u,v) \in E(H)$ the edge $\mu(e)$ has its tail in $T_{\mathrm{out}}^{\mu(u)}$ and its head in $T_{\mathrm{in}}^{\mu(v)}$
\end{itemize}
a \emph{tree-like model} of $H$ in $D$.
Further, a digraph $H$ is a \emph{butterfly minor} of a digraph $D$ if there exists a tree-like model of $H$ in $D$.\footnote{
For finite digraphs, butterfly minors are commonly defined by sequences of butterfly contractions.
Amiri, Kawarabayashi, Kreutzer and Wollan \cite{amiri2016erdos}*{Lemma 3.2} showed that, for finite digraphs, this definition coincides with our definition.
For infinite digraphs, constructing minors by sequences of butterfly contractions (with a suitable definition for limit steps) generalises our definition of butterfly minors: $\mu(v)$ can also be a digraph whose vertices have out-degree precisely one and no edge $\mu(e)$ has tail in $\mu(v)$, and, similarly, $\mu(v)$ can be a digraph whose vertices have in-degree precisely one and no edge $\mu(e)$ has head in $\mu(v)$ \cite{infinite}*{Section 2.4}.
However, in both cases there is generally no vertex in $\mu(v)$ that can reach all edges leaving $\mu(v)$ and that can be reached from all edges entering $\mu(v)$, which is a fundamental property of butterfly minors in finite digraphs.
}

The union of an out-ray $S$ rooted at $r$ and edges $(e_v)_{v \in V(S) \setminus \{r\}}$, where $e_v$ has head $r$ and tail $v$, is called a \emph{dominated directed ray}.
Furthermore, the union of an in-ray $S$ rooted at $r$ and edges $(e_v)_{v \in V(S) \setminus \{r\}}$, where $e_v$ has head $v$ and tail $r$, is also called a \emph{dominated directed ray}.

\begin{theorem}\cite{infinite}*{Corollary 1.4} \label{cor:infinite_star_comb}
	Every infinite strongly connected digraph contains either $\mathcal{D}(K_{1, \infty})$, $\mathcal{D}(S)$ for an undirected ray $S$ or a dominated directed ray as a butterfly minor.
\end{theorem}

\begin{proof}[Proof of~\cref{lem:thick_component}]
	Let $C$ be an infinite strong component of $D^\infty(I)$.
	The strong component $C$ contains a butterfly minor $M$ that is $\mathcal{D}(K_{1, \infty}), \mathcal{D}(S)$ for an undirected ray $S$ or a dominated directed ray by \cref{cor:infinite_star_comb}.
	
	Let $\mu$ be a tree-like model of $M$ in $D^\infty(I)$.
	For each vertex $v \in V(M)$ let $\sigma(v)$ be the common root of $T_{\mathrm{in}}^{\mu(v)}$ and $T_{\mathrm{out}}^{\mu(v)}$.
	Note that for every $vw \in E(M)$ there exists a $\sigma(v)$--$\sigma(w)$~path in $D^\infty(I)$ that is internally disjoint to $\bigcup_{x \in V(M)} \sigma(x)$.
	This implies that there are infinitely many disjoint $R_{\sigma(v)}$--$R_{\sigma(w)}$~paths that are internally disjoint to $\bigcup_{x \in V(M)} R_{\sigma(x)}$.
	
	If $M$ is $\mathcal{D}(S)$ for an undirected ray $S$, then we can construct a subdivision of the bidirected quarter-grid whose vertical rays are $(R_{\sigma(x)})_{x \in V(M)}$ by adding the desired $\bigcup_{x \in V(M)} R_{\sigma(x)}$-paths recursively.
	If $M$ is $\mathcal{D}(K_{1, \infty})$, let $c$ be the central vertex of $K_{1, \infty}$.
	We can construct a subdivision of the bidirected quarter-grid whose vertical rays are $(R_{\sigma(x)})_{x \in V(M) \setminus \{c\}}$ by routing the desired $\bigcup_{x \in V(M) \setminus \{c\}} R_{\sigma(x)}$--paths along $R_{\sigma(c)}$.
	If $M$ is a dominated directed ray, then we can construct a subdivision of a cyclic quarter-grid whose vertical rays are $(R_{\sigma(x)})_{x \in V(M)}$ by adding the desired $\bigcup_{x \in V(M)} R_{\sigma(x)}$-paths recursively.
\end{proof}

\section{In-rays and out-rays}\label{sec:rays}
In this section we prove another sufficient condition for~\cref{thm:main}.
We show that~\cref{thm:main} holds true if there exists $J \subseteq I$ such that $D^\infty(J)$ contains an in-ray or an out-ray.

\begin{lemma}\label{lem:in_rays}
	Let $(R_i)_{i \in I}$ be an infinite family of disjoint equivalent out-rays in a digraph $D$.
	If $D^\infty(I)$ contains an in-ray, then $D$ contains a subdivision $D'$ of either a bidirected quarter-grid or a cyclic quarter-grid such that each vertical ray of $D'$ is an element of $(R_i)_{i \in I}$.
\end{lemma}

\begin{proof}
	For simplicity, we assume without loss of generality that $\N \subseteq I$ and that $(n+1) n \in E(D^\infty(I))$ for every $n \in \N$.
	This implies $(n+1) n \in E(D^\infty(\N))$.
	Let $C$ be the strong component of $D^\infty(\N)$ containing $1$.
	By~\cref{lem:thick_component} we can assume that $C$ is finite.
	Since every element of $\N$ can reach $1$ in $D^\infty(\N)$, there exists no edge in $E(D^\infty(\N))$ with tail in $V(C)$ and head in $\N \setminus V(C)$.
	Thus, by~\cref{lem:out-degree}, there exists a vertex $\beta(1) \in V(C)$ of infinite out-degree in $D(\N)$.
	Let $\mathcal{N}$ be the out-neighbourhood of $\beta(1)$ in $D(\N)$.
	
	We construct recursively a subdivision of a descending cyclic quarter-grid whose first vertical out-ray is $R_{\beta(1)}$ and whose other vertical out-rays are in $(R_{k})_{k > \beta(1)}$.
	We assume that we fixed the first $m$ vertical out-rays $R_{\beta(1)}, \dots, R_{\beta(m)}$ and the first $m-1$ girders $G_2, \dots, G_m$ of the desired subdivision of a cyclic quarter-grid.
	
	For every $x \in \mathcal{N}$ we fix an $R_{\beta(1)}$--$R_{x}$~path $O_x$ in $S$ that is internally disjoint to $\bigcup_{n \in \N} R_n$, where $S$ is the out-ray used for the construction of $D(\N)$.
	Since the paths $(O_x)_{x \in \mathcal{N}}$ are contained in $S$ and as they are internally disjoint to $\bigcup_{n \in \N} R_n$, they are disjoint.
	Thus all but finitely many paths $(O_x)_{x \in \mathcal{N}}$ avoid $\bigcup_{j \in [m] \setminus \{1\}} \Down{G_j}_\N$.
	We pick $\beta(m+1) \in \mathcal{N} \setminus [\beta(m)]$ such that $R_{\beta(m+1)}$ and $O_{\beta(m+1)}$ avoid $\bigcup_{j \in [m] \setminus \{1\}} \Down{G_j}_\N$.
	
	It remains to construct a girder $G_{m+1}$.
	We consider a $\beta(m+1), \beta(m+1) -1, \dots, \beta(1)$--staircase $Q$ avoiding $\bigcup_{j \in [m] \setminus \{1\}} \Down{G_j}_\N \cup \Down{O_{\beta(m+1)}}_\N$, which exists by~\cref{prop:zickzack}.
	Then the concatenation of the path $O_{\beta(m+1)}$, some subpath of $R_{\beta(m+1)}$ and the path $Q$ forms the desired girder $G_{m+1}$.
	This completes the construction and finishes the proof.
\end{proof}

\begin{lemma}\label{lem:out_rays}
	Let $(R_i)_{i \in I}$ be an infinite family of disjoint equivalent out-rays in a digraph $D$.
	If $D^\infty(I)$ contains an out-ray, then $D$ contains a subdivision $D'$ of either a bidirected quarter-grid or a cyclic quarter-grid such that each vertical ray of $D'$ is an element of $(R_i)_{i \in I}$.
\end{lemma}

The proof of~\cref{lem:out_rays} follows the proof of~\cref{lem:in_rays} but constructs a subdivision of a cyclic quarter-grid that is ascending rather than descending.
Since the arches are terminal segments of the girders in ascending cyclic quarter-grids, we have to construct the girders carefully.

\begin{proof}
	For simplicity, we assume without loss of generality that $\N \subseteq I$ and that $n (n+1) \in E(D^\infty(I))$ for every $n \in \N$.
	
	\begin{claim}\label{clm:special_out_ray}
		There exist an out-ray $\hat{S}$ in $D$, a strictly increasing sequence $(\alpha(n))_{n \in \N}$ of natural numbers and a strictly $\leq_{\hat{S}}$-increasing sequence $(x_n y_n)_{n \in \N}$ of edges in $\hat{S}$ with $x_n y_n \in E(R_{\alpha(n)})$ such that for every $n > 1$
		\begin{enumerate}[label=(\alph*)]
			\item\label{itm:construction_ray_1n} a terminal segment $P_n$ of $y_{n -1} \hat{S} x_n$ is an $ \alpha(1), \alpha(1) + 1, \dots, \alpha(n)$--staircase,
			\item\label{itm:construction_ray_2n} $y_n \hat{S}$ avoids $\Down{\hat{S} x_n}_\N$, and
			\item\label{itm:construction_ray_3n} $x_n$ is $\leq_{\hat{S}}$-minimal in $\hat{S} \cap R_{\alpha(n)}$.
		\end{enumerate}
		
	\end{claim}

\begin{center}
	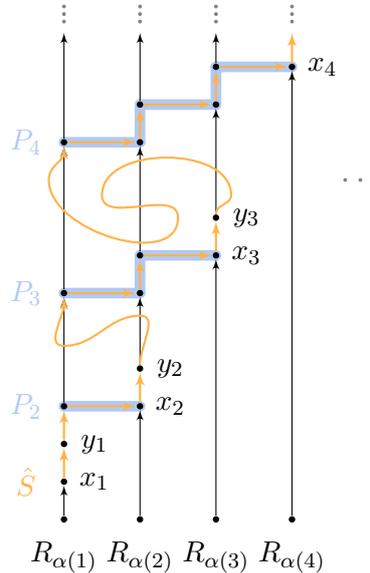
\begin{figure}[ht]
		\begin{tikzpicture}
			
			
			\foreach \x/\y in {0/3,0/6,1/7,0/10,1/11,2/12}{
				\draw[CornflowerBlue!50, line width=4pt, line cap=round] (\x,0.5*\y) to ({\x+1}, 0.5*\y);
			}
			\foreach \x/\y in {1/6,1/10,2/11}{
				\draw[CornflowerBlue!50, line width=4pt, line cap=round] (\x,0.5*\y) to ({\x}, {0.5*(\y+1)});
			}
			
			
			\foreach \y/\z in {0/1,3/6,6/10,10/13}{
				\filldraw (0,{0.5*\y}) circle (1pt);
				\draw[edge] (0,{0.5*\y+0.05}) to (0,{0.5*(\z)-0.05});
			}
			\foreach \y/\z in {1/2,2/3}{
				\filldraw (0,{0.5*\y}) circle (1pt);
				\draw[edge, PastelOrange, thick] (0,{0.5*\y+0.05}) to (0,{0.5*(\z)-0.05});
			}
			
			\foreach \y/\z in {0/3,4/6,7/10,11/13}{
				\filldraw (1,{0.5*\y}) circle (1pt);
				\draw[edge] (1,{0.5*\y+0.05}) to (1,{0.5*(\z)-0.05});
			}
			
			\foreach \y/\z in {3/4,6/7,10/11}{
				\filldraw (1,{0.5*\y}) circle (1pt);
				\draw[edge,PastelOrange, thick] (1,{0.5*\y+0.05}) to (1,{0.5*(\z)-0.05});
			}
			
			\foreach \y/\z in {0/7,8/11,12/13}{
				\filldraw (2,{0.5*\y}) circle (1pt);
				\draw[edge] (2,{0.5*\y+0.05}) to (2,{0.5*(\z)-0.05});
			}
			\foreach \y/\z in {7/8,11/12}{
				\filldraw (2,{0.5*\y}) circle (1pt);
				\draw[edge, PastelOrange, thick] (2,{0.5*\y+0.05}) to (2,{0.5*(\z)-0.05});
			}

			\foreach \y/\z in {0/12}{
				\filldraw (3,{0.5*\y}) circle (1pt);
				\draw[edge] (3,{0.5*\y+0.05}) to (3,{0.5*(\z)-0.05});
			}
			\foreach \y/\z in {12/13}{
				\filldraw (3,{0.5*\y}) circle (1pt);
				\draw[edge, PastelOrange, thick] (3,{0.5*\y+0.05}) to (3,{0.5*(\z)-0.05});
			}
			
			\draw[edge, PastelOrange, thick, smooth] (2,4.05) to[out=90, in=-90] (2.2,4.3) to[out=90, in=90] (0.5,4.3) to[out=-90, in=90] (1.5,4) to[out=-90, in=-90] (-0.2,4.5) to[out=90,in=-90](0,4.95);
			\draw[edge, PastelOrange, thick, smooth] (1,2.05) to[out=90,in=0] (1,2.7) to[out=-180,in=0] (0,2.3) to[out=-180, in=-90] (0,2.95);
			
			\foreach \x/\y in {0/3,0/6,1/7,0/10,1/11,2/12}{
				\draw[PastelOrange,edge, thick] (\x+0.05,0.5*\y) to ({\x+0.95}, 0.5*\y);
			}
			
			
			\foreach \x in {0,1,2,3}{
				\filldraw[gray] (\x,6.6) circle (0.5pt);
				\filldraw[gray] (\x,6.7) circle (0.5pt);
				\filldraw[gray] (\x,6.8) circle (0.5pt);
			}
			
			
			\filldraw[gray] (3.7,4.5) circle (0.5pt);
			\filldraw[gray] (3.9,4.5) circle (0.5pt);
			\filldraw[gray] (4.1,4.5) circle (0.5pt);
			
			\draw[] (0.4,0.5) node {$x_1$};
			\draw[] (0.4,1) node {$y_1$};
			\draw[] (1.4,1.5) node {$x_2$};
			\draw[] (1.4,2) node {$y_2$};
			\draw[] (2.4,3.5) node {$x_3$};
			\draw[] (2.4,4) node {$y_3$};
			\draw[] (3.4,6) node {$x_4$};
			
			\draw[PastelOrange] (-0.5,0.5) node {$\hat{S}$};
			\draw[CornflowerBlue!50] (-0.5,1.5) node {$P_2$};
			\draw[CornflowerBlue!50] (-0.5,3) node {$P_3$};
			\draw[CornflowerBlue!50] (-0.5,5) node {$P_4$};
			
			\foreach \x in {1,2,3,4}{
			\draw[] (\x-1,-0.5) node {$R_{\alpha(\x)}$};
			}
		\end{tikzpicture}
	\caption{An out ray $\hat{S}$ as in~\Cref{clm:special_out_ray}.}
\label{fig:construction_out_ray}
	\end{figure}
\end{center}

	\begin{claimproof}
		We pick recursively the desired strictly increasing sequence of natural numbers $(\alpha(n))_{n \in \N}$, the desired edges $x_n y_n \in E(R_{\alpha(n)})$ with $x_n <_{R_{\alpha(n)}} y_n$ and a strictly increasing sequence of paths $(\hat{S}_n)_{n \in \N}$ such that for every $n > 1$
		\begin{itemize}
			\item $\hat{S}_n$ starts in $x_1$,
			\item $\hat{S}_n$ ends in the edges $x_ny_n$,
			\item a terminal segment of $\hat{S}_n x_n$ is an $ \alpha(1), \alpha(1) + 1, \dots, \alpha(n)$--staircase,
			\item $y_{n-1} \hat{S}_n$ avoids $\Down{\hat{S}_n x_{n-1}}_\N$, and
			\item $x_n$ is $\leq_{\hat{S}_n}$-minimal in $\hat{S}_n \cap R_{\alpha(n)}$.
		\end{itemize}
		Then the out-ray $\hat{S}:= \bigcup_{n \in \N} \hat{S}_n$ is as desired.
		
		Let $x_1 y_1$ be some edge of $R_1$ with $x_1 \leq_{R_1} y_1$, let $\hat{S}_1:= x_1 R_1 y_1$ and set $\alpha(1):= 1$.
		We assume that $x_n y_n$ for $n \in [m]$, $(\alpha(n))_{n \in [m]}$ and $\hat{S}_m$ have been defined for some $m \in \N$. See~\Cref{fig:construction_out_ray}.
		
		First we pick some $R_{\alpha(m)}$--$R_{\alpha(1)}$~path $O$ that avoids $\Down{\hat{S}_m}_\N$, which exists since the rays in $(R_i)_{i \in I}$ are equivalent.
		Second, let $\alpha(m+1) \in \N \setminus [\alpha(m)]$ be such that $(\hat{S}_m \cup O) \cap R_{\alpha(m+1)} = \emptyset$.
		
		Third, let $O'$ be an $ \alpha(1), \alpha(1) + 1, \dots, \alpha(m+1)$--staircase avoiding $\Down{\hat{S}_m}_\N \cup \Down{O}_\N$, which exists by~\cref{prop:zickzack}.
		Let $x_{m+1}$ be the terminal vertex of $O'$ and $y_{m+1}$ be the vertex succeeding $x_{m+1}$ in $\leq_{R_{\alpha(m+1)}}$.
		Then the concatenation of $\hat{S}_m$, some path in $R_{\alpha(m)}$, the path $O$, some path in $R_{\alpha(1)}$, $O'$ and the edges $x_{m+1}y_{m+1}$ forms the desired path $\hat{S}_{m+1}$.
		This completes the construction.
	\end{claimproof}
	
	Let $\hat{S}$ be as in~\Cref{clm:special_out_ray}.
	Note that $\hat{S}$ intersects $R_{\alpha(n)}$ infinitely often for every $n \in \N$ by~\cref{itm:construction_ray_1n}.
	From now on we consider $D(\mathrm{A})$ and $D^{\infty}(\mathrm{A})$ with respect to the out-ray $\hat{S}$, where $\mathrm{A}:= \{\alpha(n): n \in \N \}$.
	
	Let $C$ be the strong component of $D^\infty(\mathrm{A})$ containing $\alpha(1)$.
	By~\cref{lem:thick_component} we can assume that $C$ is finite.
	Note that $\alpha(n) \alpha(n+1) \in E(D^\infty(\mathrm{A}))$ for every $n \in \N$ by construction of $\hat{S}$. Thus there exists no edge in $E(D^\infty(\mathrm{A}))$ with tail in $\mathrm{A} \setminus V(C)$ and head in $V(C)$.
	This implies that there exists a vertex $\alpha(\beta(1)) \in V(C)$ of infinite in-degree in $D(\mathrm{A})$ by \cref{lem:out-degree}.
	
	We construct recursively a subdivision of an ascending cyclic quarter-grid whose first vertical out-ray is $R_{\alpha(\beta(1))}$ and whose other vertical out-rays are in $(R_{\alpha(k)})_{k > \beta(1)}$.
	We assume that we fixed the vertical out-rays $R_{\alpha(\beta(1))}, \dots, R_{\alpha(\beta(m))}$ and the girders $G_2, \dots, G_m$ of the desired subdivision of a cyclic quarter-grid.
	
	All but finitely many of the staircases $(P_n)_{n \in \N}$, as defined in~\Cref{clm:special_out_ray}, avoid the set $\bigcup_{j \in [m] \setminus \{1\}} \Down{G_j}_\N$.
	Thus there exists $\beta(m+1)>\beta(m)$ such that $\alpha(\beta(m+1)) \in \mathrm{N}_{D(\mathrm{A})}^{\text{in}}(\alpha(\beta(1)))$ and $P_{\alpha(\beta(m+1))}$, $R_{\alpha(\beta(m+1))}$ avoid $\bigcup_{j \in [m] \setminus \{1\}} \Down{G_j}_\N$.
	
	It remains to construct the girder $G_{m+1}$.
	Let $Q$ be some $R_{\alpha(\beta(m+1))}$--$R_{\alpha(\beta(1))}$~path in $\hat{S}$ that is internally disjoint to $\bigcup_{i \in \mathrm{A}} R_i$, which exists since $\alpha(\beta(m+1)) \alpha(\beta(1)) \in E(D(\mathrm{A}))$.
	Since $Q$ starts in $R_{\alpha(\beta(m+1))}$, $Q \subseteq x_{\beta(m+1)} \hat{S}$ by~\cref{itm:construction_ray_3n}.
	Thus $Q \subseteq y_{\beta(m+1)} \hat{S}$ as $x_{\beta(m+1)} y_{\beta(m+1)} \in E(R_{\alpha(\beta(m+1))})$.
	This implies $Q$ avoids $\Down{P_{\alpha(\beta(m+1))}}_\N \subseteq \Down{\hat{S} x_{\alpha(\beta(m+1))}}_\N$ by~\cref{itm:construction_ray_2n} and since $P_{\alpha(\beta(m+1))} \subseteq \hat{S} x_{\alpha(\beta(m+1))}$.
	Thus the concatenation of the $ \alpha(\beta(1)), \alpha(\beta(1)) + 1, \dots, \alpha(\beta(m+1))$--staircase $P_{\alpha(\beta(m+1))}$, some path in $R_{\alpha(\beta(m+1))}$ and the path $Q$ forms the desired girder $G_{m+1}$.
	This finishes the construction of the subdivision of a cyclic quarter-grid whose vertical out-rays are in $(R_i)_{i \in I}$.
\end{proof}

\section{Suitable subfamilies}\label{sec:subfamilies}

Towards the proof of~\cref{thm:main} we investigate in this section the structure of $D$ if there does not exist $J \subseteq I$ such that $D^\infty(J)$ contains an infinite strong component, an in-ray or an out-ray.
More precisely, we show in~\Cref{lem:out_structure,lem:in_structure} the existence of certain families of $\bigcup_{i \in I} R_i$--paths respecting a given linear order on $I$.

\begin{lemma}\label{lem:out_structure}
	Let $(R_n)_{n \in \N}$ be either a family of disjoint equivalent out-rays or a family of disjoint equivalent in-rays in a digraph $D$.
	Then there exists an infinite subset $I \subseteq \N$ such that
	\begin{itemize}
		\item $D^\infty(I)$ has an infinite strong component,
		\item $D^\infty(I)$ contains an out-ray or an in-ray, or
		\item there is a family $(P_{i,j})_{i,j \in I: i < j}$ of paths, where $P_{i,j}$ is a $R_i$--$R_j$~path that is internally disjoint to $\bigcup_{n \in I} R_n$ such that the elements of $(P_{i,j})_{j \in I \setminus [i]}$ are disjoint for every $i \in I$.
	\end{itemize}
\end{lemma}
For the proof of~\cref{lem:out_structure} we need the following property of digraphs:
\begin{proposition}\label{rem:almost_out_degree}
	Let $D$ be an infinite digraph such that all strong components of $D$ are finite and for all but finitely many strong components $C$ of $D$ there exists an edge with tail in $V(C)$ and head in $V(D) \setminus V(C)$.
	Then $D$ contains an in-ray, an out-ray or a vertex of infinite in-degree.
\end{proposition}
\begin{proof}
	If there exists a vertex $v \in V(D)$ that can be reached by infinitely many vertices, then we can construct recursively an infinite in-arborescence $T$ with root $v$.
	The in-arborescence $T$ contains either an in-ray or a vertex of infinite in-degree.
	Thus we can assume that each vertex of $D$ can only be reached by finitely many vertices.
	
	Let $\mathcal{C}$ be the (finite) set of (finite) strong components $C$ for which there does not exist an edge with tail in $V(C)$ and head in $V(D) \setminus V(C)$.
	Further let $x$ be some vertex of $D$ that cannot reach a vertex of $\bigcup \mathcal{C}$.
	
	We construct an out-ray in $D$ by building an increasing sequence of paths $(P_n)_{n \in \N}$ starting in $x$ such that $P_n$ intersects the strong component containing its final vertex only in one vertex.
	Set $P_1:= \{x\}$.
	We assume that $P_n$ has been defined for some $n \in \N$.
	Let $C$ be the strong component containing the final vertex $x_n$ of $P_n$.
	By assumption $C \notin \mathcal{C}$.
	Thus there exists an edge $yz$ with tail in $V(C)$ and head in $V(D) \setminus V(C)$.
	Let $P_{n+1}:= P_n x_n Q yz$, where $Q$ is some $x_n$--$y$~path in $C$.
	Note that $z \notin V(P_n)$ since $C$ is a strong component of $D$.
	This completes the construction of the out-ray $\bigcup_{n \in \N} P_n$ in $D$ and finishes the proof.
\end{proof}

\begin{proof}[Proof of~\cref{lem:out_structure}]
	We assume that the digraph $D^\infty(I)$ does not contain an out-ray, an in-ray nor an infinite strong component for every infinite subset $I \subseteq \N$.
	
	\begin{claim} \label{clm:thick_in_degree_or_paths}
		There exists an infinite subset $J \subseteq \N$ such that
		\begin{itemize}
			\item $D^\infty(J)$ contains a vertex of infinite in-degree, or
			\item $D$ contains a family $(P_{i,j})_{i, j \in J: i < j}$ of paths, where $P_{i,j}$ is an $R_i$--$R_j$~path that is internally disjoint to $\bigcup_{n \in J} R_n$, such that the elements of $(P_{i,j})_{j \in J \setminus [i]}$ are disjoint for every $i \in J$.
		\end{itemize}
	\end{claim}

	\begin{claimproof}
	We assume that $D^\infty(J)$ does not contain a vertex of infinite in-degree for every infinite subset $J \subseteq I$.
	We begin by constructing recursively a strictly increasing sequence $(j_0, j_1, j_2, j_3, \dots)$ of natural numbers and a decreasing sequence $\N \supseteq I_1 \supseteq I_2 \supseteq I_3 \supseteq \dots$ of infinite sets such that for every $n \in \N$
	\begin{itemize}
		\item $j_1, \dots, j_n \notin I_n$,
		\item $j_{n+1} \in I_n$, and
		\item $I_n \subseteq \mathrm{N}_{D(K_{n-1})}^{\text{out}}(j_n)$,
	\end{itemize}
	where $K_{n-1}:= \{j_1, \dots, j_{n-1}\} \cup I_{n-1}$.
	
	Set $j_0:= 0$ and $I_0:= \N$.
	We assume that $j_1, \dots, j_{n-1}$ and $I_{n -1}$ have been defined for some $n \in \N$.
	If all elements of $I_{n-1} \setminus [j_{n-1}]$ have finite out-degree in $D(K_{n-1})$, then for every (finite) strong component $C$ of $D^\infty(K_{n-1})$ that avoids $[j_{n-1}]$ there exists an edge in $D^\infty(K_{n-1})$ with tail in $V(C)$ and head in $(K_{n-1}) \setminus V(C)$ by~\cref{lem:out-degree}.
	By~\cref{rem:almost_out_degree}, $D^\infty(K_{n-1})$ contains an in-ray, an out-ray or a vertex of infinite in-degree, contradicting our assumptions.
	
	Thus there exists $j_n \in I_{n-1} \setminus [j_{n-1}]$ that has infinite out-degree in $D(K_{n-1})$.
	We set $I_{n}:= I_{n-1} \cap \mathrm{N}_{D(K_{n-1})}^{\text{out}}(j_n)$.
	Then $j_n$ and $I_n$ are as desired, which completes the recursion.
	
	Finally, we set $J:= \{j_n : n \in \N \}$.
	Note that for every $n < m \in \N$ we have $j_m \in I_{m-1} \cap J \subseteq I_n \cap J \subseteq \mathrm{N}_{D(K_{n-1})}^{\text{out}}(j_n) \cap J \subseteq \mathrm{N}_{D(J)}^{\text{out}}(j_n)$.
	Thus there exists an $R_{j_n}$--$R_{j_m}$~path $P_{j_n,j_m}$ in $S$ that is internally disjoint to $\bigcup_{i \in J} R_i$.
	Since each path $(P_{i,j})_{i,j \in J: i < j}$ is a subpath of $S$, the elements of $(P_{i,j})_{j \in J \setminus [i]}$ are disjoint for every $i \in J$.
	Thus $(P_{i,j})_{i,j \in J: i<j}$ is indeed the desired family of paths.
\end{claimproof}

	By~\Cref{clm:thick_in_degree_or_paths}, we can assume that there is $J \subseteq \N$ such that $D^\infty(J)$ contains a vertex $\ell$ of infinite in-degree.
	We consider $J':= \{\ell\} \cup N_{D^\infty(J)}^{\text{in}}(\ell)$.
	Let $C$ be the (finite) strong component of $D^\infty(J')$ containing $\ell$.
	Note that there is no edge in $D^\infty(J')$ with tail in $V(C)$ and head in $J' \setminus V(C)$, since every element of $J' \setminus \{\ell\}$ can reach $C$ in $D^\infty(J')$.
	Thus there exists a vertex $k \in V(C)$ of infinite out-degree in $D(J')$ by~\cref{lem:out-degree}.
	
	We set $J'':= N_{D(J')}^{\text{out}}(k) \setminus V(C)$.
	We remark that for every $n \in J''$ we have $n \ell \in E(D^\infty(V(C) \cup J''))$ and $k n \in E(D(V(C) \cup J''))$ and further that there exists an $\ell$--$k$~path in $C \subseteq D^\infty(V(C) \cup J'')$.
	See~\cref{fig:construction_path}.
	
	\begin{center}
		\begin{figure}[ht]
			\begin{tikzpicture}
				\draw[LavenderMagenta] (1,2) ellipse (2 cm and 1 cm);
				\draw[LavenderMagenta] (-2,2) node {$C$};
				\draw[] (-2,0) node {$J''$};
				
				\filldraw (0,2) circle (0.5pt);
				\filldraw (2,2) circle (0.5pt);
				\draw (0,2.5) node {$\ell$};
				\draw (2,2.5) node {$k$};
				\foreach \x in {-1,-0.5,...,2}{
					\filldraw (\x,0) circle (0.5pt);
				}
				\foreach \x in {-1,-0.5,...,2}{
					\draw[edge, thick, shorten <=5pt, shorten >=5pt] (\x,0) to (0,2);
					\draw[edge, dashed, shorten <=5pt, shorten >=5pt] (2,2) to (\x,0);
				}
				
				\filldraw[opacity=0.5,thick] (2.5,0) circle (0.5pt);
				\draw[edge, opacity=0.5, thick, shorten <=5pt, shorten >=5pt] (2.5,0) to (0,2);
				\draw[edge, opacity=0.5, dashed, shorten <=5pt, shorten >=5pt] (2,2) to (2.5,0);
				
				\draw[->, > = latex', thick, decorate, decoration={snake, amplitude=1mm, segment length=5mm}] (0.1,2) -- (1.9,2);
			\end{tikzpicture}
		\caption{Edges of $D(V(C) \cup J'')$ (dashed) and $D^\infty(V(C) \cup J'')$ (thick) in the proof of~\cref{lem:out_structure}.}
			\label{fig:construction_path}
		\end{figure}
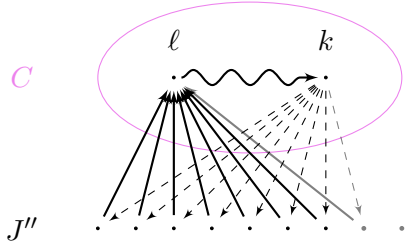
	\end{center}
	
	\begin{claim}
		The digraph $D$ contains a family $(P_{i,j})_{i,j \in I: i < j}$ of paths, where $P_{i,j}$ is a $R_i$--$R_j$~path that is internally disjoint to $\bigcup_{n \in I} R_n$, such that the elements of $(P_{i,j})_{j \in I \setminus [i]}$ are disjoint for every $i \in I$.
	\end{claim}
	
	\begin{claimproof}
	We construct recursively a strictly increasing sequence $(I_m)_{m \in \N}$ of subsets of $J''$ and the desired family $(P_{i,j})_{i,j \in I: i < j}$ of paths, where $I:= \bigcup_{m \in \N} I_m$.
	We assume that we fixed a finite subset $I_m \subseteq J''$ and constructed a family $(P_{i,j})_{i,j \in I_m: i < j}$ of paths.
	
	\begin{description}
		\item[If $(R_n)_{n \in V(C)}$ is a family of out-rays]
			Firstly, we pick, for every $i \in I_m$, an $R_i$--$R_k$~path $Q_i$ that is disjoint from $\bigcup_{i, j \in I_m: i < j} \Down{P_{i,j}}_\N$ and internally disjoint from $\bigcup_{j \in J''} R_j$.
			Such paths exist by~\cref{prop:zickzack} since there is an $i$--$k$~path in $D^\infty(V(C) \cup J'')$ that is internally disjoint to $J''$.
			
			As $k$ has infinite out-degree in $D(V(C) \cup J'')$, there exists $x \in J'' \setminus \bigcup_{i \in I_m} [i]$ such that there is an $R_k$--$R_x$~path $O$ that is disjoint from $\bigcup_{i, j \in I_m: i < j} \Down{P_{i,j}}_\N \cup \bigcup_{i \in I_m}\Down{Q_i}_\N$ and internally disjoint from $\bigcup_{j \in J''} R_j$.
			Then $Q_i \cup R_k \cup O$ contains an $R_i$--$R_x$~path $P_{i,x}$ for every $i \in I_m$.
			Note that $P_{i,x}$ is disjoint from $\bigcup_{i,j \in I_m: i < j} P_{i,j}$ and internally disjoint from $\bigcup_{j \in J''} R_j$.
		
		\item[If $(R_n)_{n \in V(C)}$ is a family of in-rays]
			Firstly, as $k$ has infinite out-degree in $D(V(C) \cup J'')$, there exists $x \in J'' \setminus (\bigcup_{i \in I_m} [i])$ such that there is an $R_k$--$R_x$~path $O$ that is disjoint from $\bigcup_{i, j \in I_m: i < j} \Down{P_{i,j}}_\N$ and internally disjoint from $\bigcup_{j \in V(C) \cup J''} R_j$.
			
			Secondly, we pick, for every $i \in I_m$, an $R_i$--$R_k$~path $Q_i$ that is disjoint from $\bigcup_{i, j \in I_m: i < j} \Down{P_{i,j}}_\N \cup \Down{O}_\N$ and internally disjoint from $\bigcup_{j \in J''} R_j$.
			Such paths exist by~\cref{prop:zickzack} since there is an $i$--$k$~path in $D^\infty(V(C) \cup J'')$ that is internally disjoint to $J''$.
			Then $Q_i \cup R_k \cup O$ contains an $R_i$--$R_x$~path $P_{i,x}$ for every $i \in I_m$.
			Note that $P_{i,x}$ is disjoint from $\bigcup_{i, j \in I_m: i < j} P_{i,j}$ and internally disjoint from $\bigcup_{j \in J''} R_j$.
	\end{description}
	\noindent
	We set $I_{m+1}:= I_m \cup \{x\}$ and continue the recursion.
	Note that the paths $P_{i,j}$ and $P_{i,j'}$ are disjoint by construction for every $i,j,j' \in I$ with $i < j,j'$ and $j \neq j'$.
		\end{claimproof}
	
	This completes the proof.
\end{proof}

By applying~\cref{lem:out_structure} to $\Dv$ we obtain:

\begin{lemma}\label{lem:in_structure}
	Let $(R_n)_{n \in \N}$ be a family of disjoint equivalent out-rays or a family of disjoint equivalent in-rays in a digraph $D$.
	Then there exists an infinite subset $I \subseteq \N$ such that
	\begin{itemize}
		\item $D^\infty(I)$ has an infinite strong component,
		\item $D^\infty(I)$ contains an in-ray or an out-ray, or
		\item there is a family $(Q_{i,j})_{i, j \in I: i > j}$ of paths, where $Q_{i,j}$ starts in $R_i$, ends in $R_j$ and is internally disjoint to $(R_n)_{n \in I} $, such that $Q_{i,k} \cap Q_{j, k} = \emptyset$ for every $i, j, k \in I$ with $i, j > k$ and $i \neq j$. \qed
	\end{itemize}
\end{lemma}

Finally, we prove the existence of an out-ray avoiding fixed initial segments of rays in~$(R_i)_{i \in I}$.
\begin{lemma}\label{lem:finding_subfamily}
	Let $(R_i)_{i \in I}$ be an infinite family of disjoint equivalent out-rays in a digraph $D$.
	Further, let $v_i \in V(R_i)$ for every $i \in I$.
	Then there exist an infinite subset $J \subseteq I$ and an out-ray $\hat{S}$ in $D$ such that for every $j \in J$:
	\begin{itemize}
		\item $\hat{S}$ avoids the initial segment $R_j v_j$, and
		\item $\hat{S}$ intersects the out-ray $R_j$ infinitely often.
	\end{itemize}
\end{lemma}
\begin{proof}
	We construct recursively a strictly increasing sequence $I_1 \subset I_2 \subset I_3 \subset \dots$ of finite subsets of $I$ and a strictly increasing sequence $P_1 \subset P_2 \subset P_3 \subset \dots$ of paths in $D$ starting in the same vertex such that for every $m \in \N$:
	\begin{itemize}
		\item $P_m$ avoids $R_i v_i$ for every $i \in I_m$,
		\item $P_{m+1} - P_{m}$ intersects $R_i$ for every $i \in I_m$, and
		\item the end vertex $t_m$ of $P_m$ is contained in $R_x$ such that $t_m R_x \cap P_m = \{t_m\}$,
	\end{itemize}
	where $x$ is a fixed element of $I_1$.
	Pick $x \in I$ and let $t_1$ be some vertex of $R_x - R_x v_x$.
	We set $I_1:=\{x\}$ and $P_1:=\{t_1\}$.
	We assume that $I_m$ and $P_m$ have been constructed.
	Since $V(P_m)$ is finite there exists $i \in I \setminus I_m$ such that $R_i v_i$ avoids $P_m$.
	We set $I_{m +1} := I_m \cup \{i\}$.
	
	\begin{claim}
		There exists a path $Q$ in $D$ such that
		\begin{itemize}
			\item the start vertex $s_{m+1}$ of $Q$ is contained in $R_x$,
			\item the end vertex $t_{m+1}$ of $Q$ is contained in $R_x$ such that $t_{m+1} R_x \cap Q = \{t_{m+1}\}$,
			\item $Q$ avoids $\Down{P_m}_{I} \cup \bigcup_{i \in I_{m+1}} V(R_i v_i)$, and
			\item $Q$ intersects $R_i$ for every $i \in I_{m+1}$,
		\end{itemize}
	\end{claim}
	\begin{claimproof}
		Let $i_1, \dots, i_k$ be a sequence of elements in $I_{m+1}$ such that each element of $I_{m+1}$ occurs at least once and such that $i_1 = i_k := x$.
		We pick recursively some $R_{i_\ell}$--$R_{i_{\ell +1}}$~path $Q_\ell$ that avoids $\Down{P_m}_{I} \cup \bigcup_{i \in I_{m+1}} V(R_i v_i)$ and $\bigcup_{j \in [\ell -1]} \Down{Q_j}_I$.
		Such a path exists since $R_{i_\ell}$ and $R_{i_{\ell +1}}$ are equivalent in $D$.
		Note that the end vertex of $Q_j$ precedes the start vertex of $Q_{j+1}$ in $R_{i_{j+1}}$ for every $j \in [k-2]$ and we set $Q_j'$ to be the path in $R_{i_{j+1}}$ starting in the end vertex of $Q_j$ and ending in the start vertex of $Q_{j+1}$.
		The concatenation of $Q_1, Q_1', Q_2, Q_2', \dots, Q_{k-2}, Q_{k-2}', Q_{k-1}$ is the desired path $Q$.
	\end{claimproof}
	\noindent
	Let $P_{m+1}$ be the concatenation of $P_m$, $t_m R_x s_{m+1}$, where $s_{m+1}$ is the start vertex of $Q$, and $Q$.
	This completes the recursion.
	Finally, we set $I:= \bigcup_{m \in \N} I_m$ and $\hat{S}:= \bigcup_{m \in \N} P_m$.
	Note that $\hat{S}$ avoids the initial segments $R_n v_n$ for every $n \in I$ and intersects every element of $(R_n)_{n \in I}$ infinitely often by construction.
\end{proof}

\section{The existence of quarter-grids}\label{sec:proof_main_theorem}
We combine the results of~\Cref{sec:component,sec:rays,sec:subfamilies} to prove:

\thmMain*
\begin{proof}
	We prove the statement for an infinite family $(R_i)_{i \in I}$ of out-rays.
	The corresponding statement for in-rays can be deduced by considering $\Dv$.
	We assume without loss of generality that $I = \N$.
	
	We apply \cref{lem:out_structure} to $(R_n)_{n \in \N}$, which provides an infinite subset $I_1 \subseteq \N$ as stated.
	If $D^\infty(I_1)$ contains an infinite strong component, an out-ray or an in-ray, we are done by \Cref{lem:thick_component,lem:in_rays,lem:out_rays}.
	Thus we can assume that there exists a family $(P_{i,j})_{i, j \in I_1: i < j}$ of paths, where $P_{i,j}$ starts in $R_i$, ends in $R_j$ and is internally disjoint to $\bigcup_{i \in I_1} R_i$, such that the paths $(P_{i,j})_{j \in I_1 \setminus [i]}$ are disjoint for every $i \in I_1$.
	
	For every $i \in I_1$ let $v(i)$ be an arbitrary vertex of $V(R_i) \setminus \bigcup_{i,j \in I_1: j < i} \Down{P_{j,i}}_\N$.
	We apply \cref{lem:finding_subfamily} to $(R_i)_{i \in I_1}$ and $(v(i))_{i \in I_1}$ to obtain an infinite subset $I_2 \subseteq I_1$ and an out-ray $\hat{S}$ that intersects every out-ray of $(R_i)_{i \in I_2}$ infinitely often and avoids all initial segments $(R_i v(i))_{i \in I_2}$.
	
	We apply~\cref{lem:in_structure} to the family $(v(i) R_i )_{i \in I_2}$ of equivalent out-rays in $\hat{D}:= \hat{S} \cup \bigcup_{i \in I_2} v(i) R_i$ to obtain an infinite subset $I_3 \subseteq I_2$ as stated.
	If $\hat{D}^\infty(I_3)$ contains an infinite strong component, an out-ray or an in-ray, then $\hat{D}$ contains a subdivision of a bidirected quarter-grid or a cyclic quarter-grid $D'$ in $\hat{D}$ whose vertical rays are elements of $(v(i) R_i )_{i \in I_3}$ by \Cref{lem:thick_component,lem:in_rays,lem:out_rays}.
	By adding some initial segments $(R_i v(i))_{i \in I_3}$ to $D'$ we obtain a subdivision of a bidirected quarter-grid or a cyclic quarter-grid in $D$ whose vertical rays are elements of $(R_i )_{i \in I_3}$.
	
	Thus we can assume that there exists a family $(Q_{i,j})_{i, j \in I_3: i > j}$ of paths in $\hat{D}$, where $Q_{i,j}$ starts in $v(i)R_i$, ends in $v(j)R_j$ and is internally disjoint to $(v(i)R_i)_{i \in I_3}$, such that $Q_{i,k} \cap Q_{j, k} = \emptyset$ for every $i, j, k \in I_3$ with $i, j > k$ and $i \neq j$.
	Note that the paths $(Q_{i,j})_{i, j \in I_3: i > j}$ are even internally disjoint to the out-rays $(R_i)_{i \in I_3}$ since they are subgraphs of $\hat{D}$.
	
	We found a subfamily $(R_i)_{i \in I_3}$ and families $(P_{i,j})_{i,j \in I_3: i < j}, (Q_{i,j})_{i, j \in I_3: i > j}$ of $(\bigcup_{i \in I_3} R_i)$-paths such that
	\begin{itemize}
		\item $P_{i,j}$ starts in $R_i$, ends in $R_j$ and is internally disjoint to $\bigcup_{i \in I_3} R_i$ for every $i,j \in I_3$ with $i < j$,
		\item $P_{i,j} \cap P_{i, k} = \emptyset$ for every $i, j, k \in I_3$ with $i < j, k$ and $j \neq k$,
		\item $Q_{i,j}$ starts in $R_i$, ends in $R_j$ and is internally disjoint to $\bigcup_{i \in I_3} R_i$ for every $i, j \in I_3$ with $i > j$,
		\item $Q_{i,k} \cap Q_{j, k} = \emptyset$ for every $i, j, k \in I_3$ with $i, j > k$ and $i \neq j$, and
		\item the end vertex of $P_{i,k}$ precedes the start vertex of $Q_{k, j}$ in $R_k$ for every $i,j, k \in I_3$ with $i,j < k$.
	\end{itemize}
	Let $I_4 \subseteq I_3$ be an infinite subset such that $I_3 \setminus I_4$ is infinite.
	\begin{claim}\label{clm:out_ray_strongly}
		There exists an out-ray $\hat{S}$ in $D$ intersecting every element of $(R_i)_{i \in I_4}$ infinitely often such that $D^\infty(I_4)$ with respect to $\hat{S}$ is strongly connected.
	\end{claim}
	\begin{claimproof}
		Let $(k_n)_{n \in \N}$ be a sequence of elements of $I_4$ such that each
		pair of distinct elements of $I_4$ appears infinitely often in $((k_n, k_{n+1}))_{n \in \N}$.
		We construct a strictly increasing sequence $(O_n)_{n \in \N}$ of paths in $D$ with the same start vertex such that for every $n \in \N$
		\begin{itemize}
			\item $O_{n+1} - O_n$ is the concatenation of a path in $R_{k_n}$ and some $R_{k_n}$--$R_{k_{n+1}}$~path that is internally disjoint to $(R_i)_{i \in I_4}$, and
			\item the final vertex $x_{n}$ of $O_n$ is in $R_{k_n}$ and satisfies $O_n \cap x_n R_{k_n} = \{x_n\}$.
		\end{itemize}
		Then the out-ray $\hat{S}:= \bigcup_{n \in \N} O_n$ is as desired.
		
		We let $x_1$ be some vertex of $R_{k_1}$ and set $O_1:= \{x_1\}$.
		Now we suppose that $O_n$ has been constructed for some $n \in \N$.
		Note that all but finitely many out-rays of $(R_i)_{i \in I_3 \setminus I_4}$ avoid the finite set $\Down{V(O_n)}_{I_4}$.
		Further, all but finitely many elements of $(P_{k_n,j} )_{j \in I_3 \setminus (I_4 \cup [k_n])}$ avoid $\Down{V(O_n)}_{I_4}$.
		Similarly, all but finitely many elements $(Q_{j, k_{n+1}} )_{j \in I_3 \setminus ( I_4 \cup [k_{n+1}])}$ avoid $\Down{V(O_n)}_{I_4}$.
		Thus there exists $m \in I_3 \setminus (I_4 \cup [k_n] \cup [k_{n+1}] )$ such that $R_m, P_{k_n,m}$ and $Q_{m, k_{n+1}} $ avoid $\Down{V(O_n)}_{I_4}$.
		Then $P_{k_n, m} \cup R_m \cup Q_{m, k_{n+1}}$ contains an $R_{k_n}$--$R_{k_{n+1}}$~path $\hat{O}_{n+1}$ that is internally disjoint to $\bigcup_{i \in I_4} R_i$.
		Let $O_{n+1}$ be the concatenation of $O_n$, some subpaths of $R_{k_n}$ and $\hat{O}_{n+1}$.
	\end{claimproof}
	
	Let $D^\infty(I_4)$ be the (strongly connected) graph with respect to the out-ray~$\hat{S}$ from~\Cref{clm:out_ray_strongly}.
	Then~\cref{lem:thick_component} applied to $D^\infty(I_4)$ shows that $D$ contains a subdivision of either a bidirected quarter-grid or a cyclic quarter-grid whose vertical out-rays are elements of $(R_i)_{i \in I}$.
\end{proof}

Now we apply~\cref{thm:main} to show:

\thmZutherEnd*

\begin{proof}
	We prove the statement for a family $(R_i)_{i \in I}$ of out-rays.
	By considering $\Dv$, we derive from this statement its counterpart for in-rays.
	By \cref{thm:main} we can assume without loss of generality that there exists a subdivision of a cyclic quarter-grid $D'$ in $D$ whose vertical out-rays are elements of $(R_i)_{i \in I}$.
	Further, we assume without loss of generality that $\N \subseteq I$ and $R_1, R_2, R_3, \dots$ are the vertical out-rays of $D'$.

	We begin by fixing a sequence $((\alpha(n), \beta(n)))_{n \in \N}$ of elements in $\{ (\ell, \ell+1), (\ell+1, \ell): \ell \in \N \}$ such that each element of $\{ (\ell, \ell+1), (\ell+1, \ell): \ell \in \N \}$ appears infinitely often and $\alpha(n), \beta(n) \leq n $ for every $n > 1$.

	We construct recursively for every $n \in \N$ a family $(S_i^n)_{i \in [n]}$ of disjoint paths in $D'$ and some $\bigcup_{i \in [n]} S_i^n$--path $P^n$ such that for every $n > 1$
	\begin{itemize}
		\item $S_i^{n-1}$ is a proper initial segment of $S_i^{n}$ for every $i \in [n - 1]$,
		\item $S_i^n$ ends in $x_i^n \in V(R_i)$ such that $x_i^n R_i \cap (\bigcup_{\ell \in [n]} P^\ell \cup \bigcup_{j \in [n]} S_j^n) = \{x_i^n\}$ for every $i \in [n]$,
		\item $P^n$ starts in $S_{\alpha(n)}^n$ and ends in $S_{\beta(n)}^n$, and
		\item $P^n$ is disjoint to $\bigcup_{\ell \in [n-1]} P^\ell$.
	\end{itemize}
	After finishing the recursion, we set $S_i:= \bigcup_{n \in \N} S_i^n$ for every $i \in \N$ and note that $(S_i)_{i \in \N}$ is a family of disjoint out-rays that are equivalent to $(R_i)_{i \in I}$.
	Then there exists $J \subseteq \N$ such that $\bigcup_{i \in \N} S_i \cup \bigcup_{\ell \in J} P^\ell$ includes the desired subdivision of a bidirected quarter-grid by the choice of $((\alpha(n), \beta(n)))_{n \in \N}$.

	Let $x_1^1$ be some vertex of $R_1$ and set $S_1^1:= \{x_1^1\}$, $P^1:= \{x_1^1\}$.
	We assume that $(S_i^{n-1})_{i \in [n-1]}$ and $(P^\ell)_{\ell \in [n-1]}$ have been defined for some $n > 1$.
	Note that all but finitely many girders of $D'$ avoid $A:= \Down{\bigcup_{i \in [n-1]} S_i^{n-1} \cup \bigcup_{\ell \in [n-1]} P^\ell}_\N$.
	If there exists a girder in $D'$ that contains an $\bigcup_{i \in \N} R_i$--path $P^n$ starting in $R_{\alpha(n)}$, ending in $R_{\beta(n)}$ and avoiding $A$, then let $S_i^n$ be some proper extension of $S_i^{n-1}$ along $R_i$ for every $i \in [n-1]$ and let $S_n^n={x_n^n}$ for some vertex $x_n^n \in V(R_n) \setminus A$ such that $\bigcup_{i \in [n]} S_i^{n}$ contains the endpoints of $P^n$.
	
	If there does not exist a girder in $D'$ that contains an $\bigcup_{i \in \N} R_i$--path $P^n$ starting in $R_{\alpha(n)}$, ending in $R_{\beta(n)}$ and avoiding $A$, then either $\beta(n)=\alpha(n) + 1$ and $D'$ is descending or $\beta(n) = \alpha(n) - 1$ and $D'$ is ascending.
	
	\begin{description}
		\item[If $\beta(n)=\alpha(n) + 1$ and $D'$ is descending] 
		See~\cref{fig:construction_bidirected_quarter_grid_in}.
		Firstly, we construct recursively disjoint paths $Q_1, \dots, Q_n$ with the following properties:
		For every $j \in [n]$, the path $Q_j$ is the union of a path in $R_j$ starting in the end vertex of $S_j^{n-1}$ if $j \in [n-1]$, some $j, \dots, 1$--staircase, some path in $R_1$ and an arch ending in some $R_{f(j)}$, where $(f(i))_{i \in [n]}$ is a strictly increasing sequence of natural numbers greater than $n$.
		Furthermore, the path $Q_j$ avoids $A \cup \bigcup_{k \in [j-1]} \Down{Q_k}_\N$ for every $j \in [n]$.
		This construction is possible since all but finitely many girders avoid $A \cup \bigcup_{k \in [j-1]} \Down{Q_k}_\N$ and $Q_1, \dots, Q_{j-1}$ avoid $R_j$.
		
		Secondly, we construct recursively disjoint paths $Q_1', \dots, Q_n'$ such that $Q_j'$ is the concatenation of some path in $R_{f(j)}$ and a $f(j), \dots, j$--staircase avoiding $A \cup \bigcup_{i \in [n-1]} \Down{Q_i}_\N \cup \bigcup_{k \in [j-1]} \Down{Q_k'}_\N$ for every $j \in [n]$.
		This is possible since all but finitely many girders avoid $A \cup \bigcup_{i \in [n-1]} \Down{Q_i}_\N \cup \bigcup_{k \in [j-1]} \Down{Q_k'}_\N$ and $Q_1, \dots, Q_{j-1}'$ avoid $R_{f(j)}$.
		
		Thirdly, we let $S_i^n$ be the concatenation of $S_i^{n-1}, Q_i$ and $Q_i'$ for $i \in [n]$.
		Note that $R_1$ contains an $S_{\alpha(n)}^n$--$S_{\beta(n)}^n$~path $P^n$.
	
		\item[If $\beta(n) = \alpha(n) - 1$ and $D'$ is ascending]
			See~\cref{fig:construction_bidirected_quarter_grid_out}.
			Firstly, we construct recursively disjoint paths $Q_1, \dots, Q_n$ such that $Q_j$ is the concatenation of some path in $R_{j}$ and a $j, \dots, j+n$--staircase avoiding $A \cup \bigcup_{k \in [n] \setminus [j]} \Down{Q_k}_\N$ for every $j \in [n]$.
			This is possible since all but finitely many girders avoid $A \cup \bigcup_{k \in [n] \setminus [j]} \Down{Q_k}_\N$ and $Q_{j+1}, \dots, Q_n$ avoid $R_{j}$.
			
			Secondly, we construct recursively a strictly increasing sequence $(f(i))_{i \in [n]}$ of natural numbers and disjoint paths $Q_1', \dots, Q_n'$ with the following properties:
			The path $Q_j'$ is the union of a path in $R_{j+n}$ starting in the end vertex of $Q_j$, some $j+n, \dots, f(j)$--staircase, some path in $R_{f(j)}$, the unique arch starting in $R_{f(j)}$, some path in $R_1$ and a $1, \dots, j$--staircase for every $j \in [n]$.
			Furthermore, the path $Q_j'$ avoids $A \cup \bigcup_{k \in [j-1]} \Down{Q_k}_\N \cup \bigcup_{k \in [n] \setminus [j]} \Down{Q_k'}_\N$ for every $j \in [n]$.
			This construction is possible since all but finitely many girders avoid $A \cup \bigcup_{k \in [j-1]} \Down{Q_k}_\N \cup \bigcup_{k \in [n] \setminus [j]} \Down{Q_k'}_\N$ and $Q_{j+1}, \dots, Q_n$ avoid $R_{j+n} \cup R_{f(j)}$.
			
	Thirdly, we set $S_i^n$ to be the concatenation of $S_i^{n-1}, Q_i$ and $Q_i'$ for $i \in [n]$.
	Note that $R_1$ contains an $S_{\alpha(n)}^n$--$S_{\beta(n)}^n$~path $P^n$.
	\end{description}
	This completes the proof.
\end{proof}

	\begin{figure}[ht]
		\centering
		\begin{subfigure}[b]{0.49\textwidth}
			\centering
		\begin{tikzpicture}
			\draw[gray!30, line width=4, line cap=round] (0,0) -- (0,2);
			\draw[gray!30, line width=4, line cap=round] (1,0) -- (1,2);
			\draw[gray!30, line width=4, line cap=round] (0,3.75) -- (0,8);
			\draw[gray!30, line width=4, line cap=round] (1,5.25) -- (1,8);						
			
			
			\draw[CornflowerBlue, line width=4, line cap=round] (0,1.5) -- (0,2.25);


			\foreach \a/\b/\c/\d in {2/7/1/7,1/7/1/7.5,1/7.5/0/7.5}{
				\draw[LavenderMagenta, line width=4, line cap=round] (\a,0.5*\b) -- (\c,0.5*\d);
			}
			
			\foreach \a/\b/\c/\d in {1/2/1/4.5,1/4.5/0/4.5}{
				\draw[PastelOrange!50, line width=4, line cap=round] (\a,0.5*\b) -- (\c,0.5*\d);
			}
			
			\draw[PastelOrange, line width=4, line cap=round] (2,5.25) -- (1,5.25);
			
			
			\foreach \y/\z in {0/1,1/2,2/4,4/4.5,4.5/7,7/7.5,7.5/10.5,10.5/11,11/14.5,14.5/15,15/16}{
				\filldraw (1,{0.5*\y}) circle (1pt);
				\draw[edge] (1,{0.5*\y+0.05}) to (1,{0.5*(\z)-0.05});
			}

				\draw[white, bend left=20, line width=0.2cm, shorten <=5pt, shorten >=5pt] ({0.05}, 1.5) to (1.95,0.5*3);
				\draw[LavenderMagenta!50, line width=4, line cap=round, bend left=20] (0,1.5) to (2,1.5);
				\draw[edge, bend left=20] ({0.05}, 0.5*3) to (2-0.05,0.5*3);
				
			
			\draw[LavenderMagenta, line width=4, line cap=round] (2,1.5) -- (2,3.5);
			
			\foreach \a/\b/\c/\d in {3/10/2/10,2/10/2/10.5}{
				\draw[PastelOrange, line width=4, line cap=round] (\a,0.5*\b) -- (\c,0.5*\d);
			}
			
			
			\foreach \y/\z in {2/3,3/4,4/6.5,6.5/7,7/10,10/10.5,10.5/14,14/14.5,14.5/16}{
				\filldraw (2,{0.5*\y}) circle (1pt);
				\draw[edge] (2,{0.5*\y+0.05}) to (2,{0.5*(\z)-0.05});
			}
			
			
				\draw[white, bend left=20, line width=0.2cm] ({0.05}, 2.75) to (2.95,2.75);
				\draw[PastelOrange!50, line width=4, line cap=round, bend left=20] (0,2.75) to (3,2.75);
				\draw[edge, bend left=20] ({0.05}, 2.75) to (3-0.05,2.75);
				
			
				\draw[PastelOrange, line width=4, line cap=round] (3,2.75) -- (3,5);

				\draw[PastelOrange!50, line width=4, line cap=round] (0,2.25) -- (0,2.75);

				\draw[LavenderMagenta!50, line width=4, line cap=round] (0,1) -- (0,1.5);
				
			
						\foreach \y/\z in {0/1,1/2,2/3,3/4.5,4.5/5.5,5.5/7.5,7.5/8.5,8.5/11,11/12,12/15,15/16}{
				\filldraw (0,{0.5*\y}) circle (1pt);
				\draw[edge] (0,{0.5*\y+0.05}) to (0,{0.5*(\z)-0.05});
			}

			\foreach \y/\z in {4.5/5.5,5.5/6.5,6.5/9.5,9.5/10,10/13.5,13.5/14,14/16}{
				\filldraw (3,{0.5*\y}) circle (1pt);
				\draw[edge] (3,{0.5*\y+0.05}) to (3,{0.5*(\z)-0.05});
			}
			
			\foreach \y/\z in {7.5/8.5,8.5/9.5,9.5/13,13/13.5,13.5/16}{
				\filldraw (4,{0.5*\y}) circle (1pt);
				\draw[edge] (4,{0.5*\y+0.05}) to (4,{0.5*(\z)-0.05});
			}
			
			\foreach \y/\z in {11/12,12/13,13/16}{
				\filldraw (5,{0.5*\y}) circle (1pt);
				\draw[edge] (5,{0.5*\y+0.05}) to (5,{0.5*(\z)-0.05});
			}
				
			\foreach \x/\y in {1/1,4/8.5,5/12}{
				\draw[white, bend left=20, line width=0.2cm, shorten <=5pt, shorten >=5pt] ({0.05}, 0.5*\y) to (\x-0.05,0.5*\y);
				\draw[edge, bend left=20] ({0.05}, 0.5*\y) to (\x-0.05,0.5*\y);
			}
			
			
			\foreach \x/\y in {0/2,0/4.5,1/4,0/7.5,1/7,2/6.5,0/11,1/10.5,2/10,3/9.5,4/13,3/13.5,2/14,1/14.5,0/15}{
				\draw[edge] ({\x+0.95}, 0.5*\y) to (\x+0.05,0.5*\y);
			}

			
			\foreach \x in {0,1,2,3,4,5}{
				\filldraw[gray] (\x,8.35) circle (0.5pt);
				\filldraw[gray] (\x,8.45) circle (0.5pt);
				\filldraw[gray] (\x,8.55) circle (0.5pt);
			}
			
			
			\filldraw[gray] (5.7,6.5) circle (0.5pt);
			\filldraw[gray] (5.9,6.5) circle (0.5pt);
			\filldraw[gray] (6.1,6.5) circle (0.5pt);
			
			\draw[LavenderMagenta!50] (-0.5,1.25) node {$Q_1$};
			\draw[CornflowerBlue] (-0.5,1.75) node {$P^2$};
			\draw[PastelOrange!50] (-0.5,2.5) node {$Q_2$};
			
			\draw[LavenderMagenta] (2.5,1.75) node {$Q_1'$};
			\draw[PastelOrange] (3.5,3) node {$Q_1'$};
		\end{tikzpicture}
		\caption{If $\beta(n)=\alpha(n) + 1$ and $D'$ is descending.}
		\label{fig:construction_bidirected_quarter_grid_in}
		\end{subfigure}
		\hfill
		\begin{subfigure}[b]{0.49\textwidth}
			\centering
					\begin{tikzpicture}
						
				\draw[gray!30, line width=4, line cap=round] (8,0) -- (8,2);
				\draw[gray!30, line width=4, line cap=round] (9,0) -- (9,2);
				\draw[gray!30, line width=4, line cap=round] (8,7) -- (8,8);
				\draw[gray!30, line width=4, line cap=round] (9,6) -- (9,8);		
						
				
				\draw[CornflowerBlue, line width=4, line cap=round] (8,5.25) -- (8,7.25);

					\foreach \a/\b/\c/\d in {9/1.5/9/6,9/6/10/6,10/6/10/6.5,10/6.5/11/6.5}{
						\draw[PastelOrange!50, line width=4, line cap=round] (\a,0.5*\b) -- (\c,0.5*\d);
					}
					\foreach \a/\b/\c/\d in {11/6.5/11/10,11/10/12/10,12/10/12/10.5,8/10.5/8/12,8/12/9/12}{
						\draw[PastelOrange, line width=4, line cap=round] (\a,0.5*\b) -- (\c,0.5*\d);
					}
					\foreach \a/\b/\c/\d in {8/1.5/8/8.5,8/8.5/9/8.5,9/8.5/9/9,9/9/10/9}{
						\draw[LavenderMagenta!50, line width=4, line cap=round] (\a,0.5*\b) -- (\c,0.5*\d);
					}
					\foreach \a/\b/\c/\d in {10/9/10/13,10/13/11/13,11/13/11/13.5,11/13.5/12/13.5,12/13.5/12/14,12/14/13/14,13/14/13/14.5}{
						\draw[LavenderMagenta, line width=4, line cap=round] (\a,0.5*\b) -- (\c,0.5*\d);
					}
					
				
				\foreach \y/\z in {0/1,1/1.5,1.5/3,3/4,4/5.5,5.5/7,7/8.5,8.5/10.5,10.5/12,12/14.5,14.5/16}{
					\filldraw (8,{0.5*\y}) circle (1pt);
					\draw[edge] (8,{0.5*\y+0.05}) to (8,{0.5*(\z)-0.05});
				}
				
				\foreach \y/\z in {0/1,1/1.5,1.5/3,3/3.5,3.5/5.5,5.5/6,6/8.5,8.5/9,9/12,12/12.5,12.5/16}{
					\filldraw (9,{0.5*\y}) circle (1pt);
					\draw[edge] (9,{0.5*\y+0.05}) to (9,{0.5*(\z)-0.05});
				}
				
				\foreach \y/\z in {2.5/3.5,3.5/4,4/6,6/6.5,6.5/9,9/9.5,9.5/12.5,12.5/13,13/16}{
					\filldraw (10,{0.5*\y}) circle (1pt);
					\draw[edge] (10,{0.5*\y+0.05}) to (10,{0.5*(\z)-0.05});
				}
				\foreach \y/\z in {5.5/6.5,6.5/7,7/9.5,9.5/10,10/13,13/13.5,13.5/16}{
					\filldraw (11,{0.5*\y}) circle (1pt);
					\draw[edge] (11,{0.5*\y+0.05}) to (11,{0.5*(\z)-0.05});
				}
				
				\foreach \y/\z in {9/10,10/10.5,10.5/13.5,13.5/14,14/16}{
					\filldraw (12,{0.5*\y}) circle (1pt);
					\draw[edge] (12,{0.5*\y+0.05}) to (12,{0.5*(\z)-0.05});
				}
				
				\foreach \y/\z in {13/14,14/14.5,14.5/16}{
					\filldraw (13,{0.5*\y}) circle (1pt);
					\draw[edge] (13,{0.5*\y+0.05}) to (13,{0.5*(\z)-0.05});
				}

				
				\foreach \x/\y in {8/1,8/3,9/3.5,8/5.5,9/6,10/6.5,8/8.5,9/9,10/9.5,11/10,8/12,9/12.5,10/13,11/13.5,12/14}{
					\draw[edge] (\x+0.05,0.5*\y) to ({\x+0.95}, 0.5*\y);
				}
				
				
				\foreach \x/\y in {9/1.5,10/4,11/7}{
					\draw[white, bend right=20, line width=0.2cm, shorten <=5pt, shorten >=5pt] (\x-0.05,0.5*\y) to ({8.05}, 0.5*\y);
					\draw[edge, bend right=20, shorten <=1pt, shorten >=1pt] (\x,0.5*\y) to ({8}, 0.5*\y);
				}
				
				\draw[white, bend right=20, line width=0.2cm, shorten <=5pt, shorten >=5pt] (12,5.25) to (8,5.25);
				\draw[PastelOrange, line width=4, bend right=20, line cap=round] (12,5.25) to (8,5.25);
				\draw[edge, bend right=20, shorten <=1pt, shorten >=1pt] (12,5.25) to (8,5.25);
				
				\draw[white, bend right=20, line width=0.2cm, shorten <=5pt, shorten >=5pt] (13,7.25) to (8,7.25);
				\draw[LavenderMagenta, line width=4, bend right=20, line cap=round] (13,7.25) to (8,7.25);
				\draw[edge, bend right=20, shorten <=1pt, shorten >=1pt] (13,7.25) to (8,7.25);
				
				\filldraw (8,5.25) circle (1pt);
				\filldraw (8,7.25) circle (1pt);
				\filldraw (12,5.25) circle (1pt);
				\filldraw (13,7.25) circle (1pt);
				
				\draw[edge] (8,5.225) to (8,5.95);
				\draw[edge] (8,7.275) to (8,7.95);
				\draw[edge] (12,5.225) to (12,6.7);
				\draw[edge] (13,7.275) to (13,7.95);
				
				
				\foreach \x in {8,9,10,11,12,13}{
					\filldraw[gray] (\x,8.35) circle (0.5pt);
					\filldraw[gray] (\x,8.45) circle (0.5pt);
					\filldraw[gray] (\x,8.55) circle (0.5pt);
				}
				
				
				\filldraw[gray] (13.7,6.5) circle (0.5pt);
				\filldraw[gray] (13.9,6.5) circle (0.5pt);
				\filldraw[gray] (14.1,6.5) circle (0.5pt);
				
				\draw[LavenderMagenta!50] (7.5,1.25) node {$Q_1$};
				\draw[CornflowerBlue] (7.5,6.5) node {$P^2$};
				\draw[PastelOrange!50] (9.5,1.25) node {$Q_2$};
				
				\draw[LavenderMagenta] (12.5,6.5) node {$Q_1'$};
				\draw[PastelOrange] (7.5,5.5) node {$Q_2'$};
			\end{tikzpicture}
			\caption{If $\beta(n)=\alpha(n) - 1$ and $D'$ is ascending.}
			\label{fig:construction_bidirected_quarter_grid_out}
		\end{subfigure}
		\caption{The construction of a subdivision of a bidirected quarter grid in the proof of~\cref{thm:zuther_end}.}
	\end{figure}
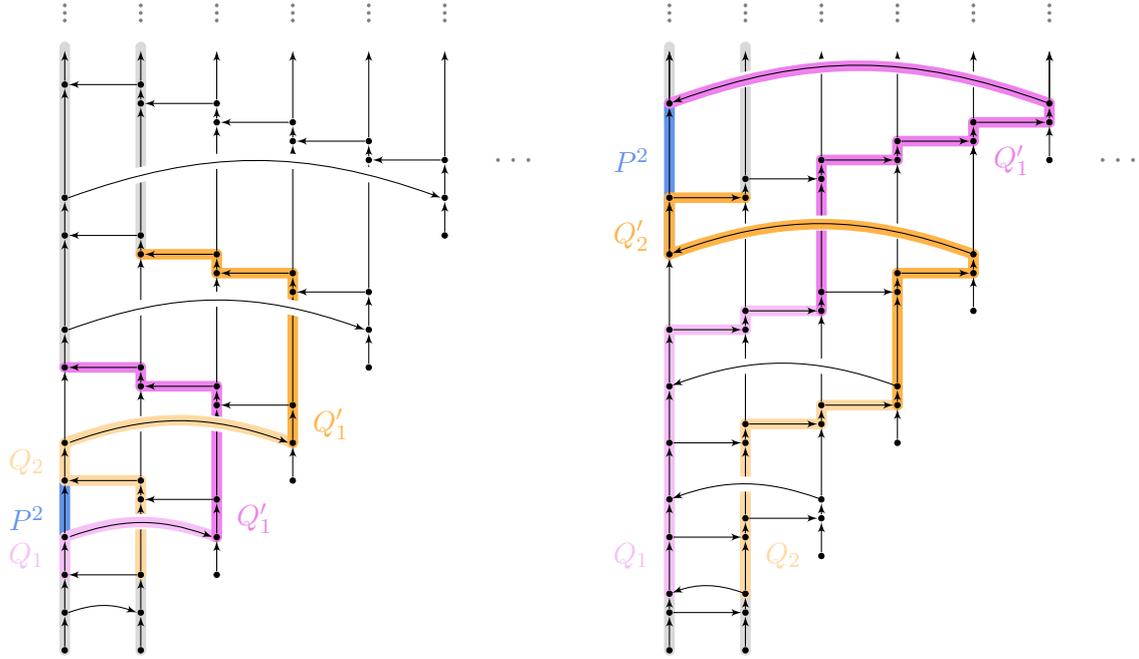

\section{The existence of necklace grids}\label{sec:necklace_grids}
Finally, we turn our attention to necklaces.
We recall the definition of necklace.
A strongly connected subgraph $N$ is a \emph{necklace} if there exists a family $(H_n)_{n \in \N}$ of finite strongly connected subgraphs such that $N = \bigcup_{\ell \in \N} H_\ell$ and $H_\ell \cap H_k \neq \emptyset$ holds if and only if $|\ell - k| \leq 1$ for every $\ell, k \in \N$.
The family $(H_\ell)_{\ell \in \N}$ is a \emph{witness} of $N$ and we define $V(H_\ell) \setminus (V(H_{\ell -1}) \cup V(H_{\ell +1}))$ to be the \emph{interior} of $H_\ell$ for every $\ell >1$.
Furthermore, for every finite set $X \subseteq V(N)$, there exists a unique infinite strong component $C$ in $N-X$ and we call $C$ a \emph{tail} of $N$.
Note that every tail of $N$ contains all but finitely many vertices of $N$.~\cite{bowler2024connectoids}

The \emph{necklace grids} are obtained from subdivisions of the bidirected and the cyclic quarter-grid, respectively, by replacing each vertical out-ray by a necklace $N$ with a fixed witness $(H_\ell^N)_{\ell \in \N}$ such that
\begin{itemize}
	\item for every odd $\ell$, there is at most one edge that is not in $E(N)$ and has head or tail in $H_\ell^N$, and
	\item for every even $\ell$, there is no edge in $E(N)$ and has head or tail in the interior of $H_\ell^N$.
\end{itemize}
The \emph{girders} of necklace grids are defined in the same way as for bidirected quarter-grids and cyclic quarter-grids.
A \emph{strong minor} of a digraph $D$ is a digraph obtained from $D$ by deletion of vertices and edges, and contraction of strongly connected subgraphs.

\thmNecklace*

\begin{proof}[Proof of~\cref{thm:necklace}]
	In this proof we construct an auxiliary digraph $A$ with the property that disjoint paths in $A$ correspond to disjoint paths in $D$ and that contains a family $\R$ of disjoint out-rays that correspond one-to-one to the necklaces of $\cN$.
	Finally, we apply \cref{thm:main} to $A$ and $\R$ to obtain a subdivision of either a bidirected quarter-grid or a cyclic quarter-grid in $A$ which corresponds to a bidirected necklace grid or a cyclic necklace grid in $D$, respectively.
	
	For every $N \in \cN$ let $(H_\ell^N)_{\ell \in \N}$ be a witness of $N$.
	We assume without loss of generality that $\cN$ is countable and further that $\cN= (N_i)_{i \in \N}$.
	Let $((j_n,k_n))_{n \in \N}$ be a sequence of pairs of distinct natural numbers such that each pair of distinct natural numbers appears infinitely often.
	
	We construct recursively a family $(P_n)_{n \in \N}$ of disjoint paths in $D$.
	We assume that $(P_i)_{i \in [n-1]}$ have been constructed.
	For each $N \in \cN$ that has been hit by some path $P_i$ for $i \in [n-1]$ let $\alpha(n,N) \in \N$ be minimum such that $\bigcup_{\ell \geq \alpha(n,N)} H_\ell^N$ avoids $P_1, \dots, P_{n - 1}$.
	For all other $N \in \cN$ set $\alpha(n,N)=0$.
	Let $P_n$ be an $N_{j_n}$--$N_{k_n}$~path in $D$ that avoids the finite set $\bigcup_{N \in \cN} \bigcup_{1 \leq \ell \leq \alpha(n,N)} V(H_\ell^N)$ and the paths $P_1, \dots, P_{n-1}$.
	This completes the construction of $(P_n)_{n \in \N}$.
	
	We consider the subgraph $\hat{D}:= \bigcup \cN \cup \bigcup_{n \in \N} P_n$.
	Note that the necklaces of $\cN$ are equivalent in $\hat{D}$.
	For every $N \in \cN$ set $\alpha(N):=\{\alpha(n,N): n \in \N \} \setminus \{0\}$.
	Note that $\alpha(N)$ is infinite and $|\alpha(n,N) - \alpha(m,N)| \geq 2$ for every $n, m \in \N$ with $n \neq m$.
	Note further that the paths $(P_n)_{n \in \N}$ avoid $\bigcup_{\ell \in \alpha(N)} H_\ell^N$.
	
	We construct a strong minor $A$ of $\hat{D}$ by applying the following steps to every $N \in \cN$ (see~\cref{fig:construction_aux_graph}):
	Firstly, for every $\ell \in \alpha(N)$ we remove all vertices of the interior of $H_\ell^N$ except that we keep a directed $H_{\ell - 1}^N$--$H_{\ell + 1}^N$~path in $H_\ell^N$.
	Secondly, we contract every strong component of $\bigcup_{\ell \in \N \setminus \alpha(N)} H_\ell^N$ to a single vertex. 
	In this way the necklace $N$ turns into an out-ray.
	
	Let $\R$ be the family of disjoint out-rays in $A$ obtained from $\cN$ in $D$.
	Note that the paths $(P_n)_{n \in \N}$ contract to trails $(T_n)_{n \in \N}$ in $A$ such that the endpoints of $T_n$ correspond to the endpoints of $P_n$.
	Thus $\R$ is a family of equivalent out-rays in $A$.
	Note further that every path $Q$ in $A$ expands to a subgraph that contains a path whose endpoints correspond to the endpoints of $Q$.
	Moreover, every two disjoint paths $Q, Q'$ expand to disjoint subgraphs of $D$.
	
	\begin{center}
		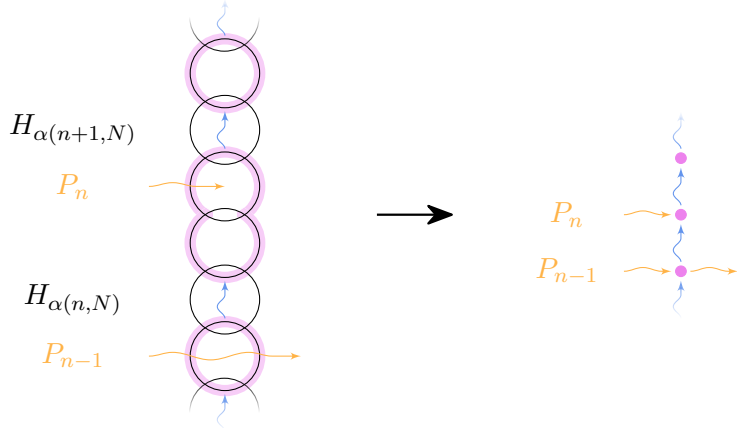
\begin{figure}[ht]
			\begin{tikzpicture}
				
		\draw[LavenderMagenta, line width=0.15cm, opacity=0.4] (0,2.25) circle (13pt);
		\filldraw[white] (0,1.5) circle (15pt);
		\draw[LavenderMagenta, line width=0.15cm, opacity=0.4] (0,1.5) circle (13pt);
		\filldraw[white] (0,2.25) circle (10.5pt);
		\filldraw[white] (-0.15,1.75) rectangle (0.15,2.0);
				
		\draw[LavenderMagenta, line width=0.15cm, opacity=0.4] (0,3.75) circle (13pt);
		\draw[LavenderMagenta, line width=0.15cm, opacity=0.4] (0,0) circle (13pt);
				
		\foreach \x in {0,0.75,...,3.75}{
		\draw[] (0,\x) circle (13pt);
			}
		
		 \draw[path fading=south] (0.45,-0.75) arc [start angle=0, end angle=180, radius=13pt];
		 \draw[path fading=north] (0.45,4.5) arc [start angle=0, end angle=-180, radius=13pt];
		 			
		\draw[->, > = latex', CornflowerBlue, decorate, decoration={snake, amplitude=1mm, segment length=5mm}] (0,0.5) -- (0,1);
		
		\draw[->, > = latex', CornflowerBlue, decorate, decoration={snake, amplitude=1mm, segment length=5mm}] (0,2.75) -- (0,3.25);
		
		\draw[->, > = latex', CornflowerBlue, path fading=south, decorate, decoration={snake, amplitude=1mm, segment length=5mm}] (0,-1) -- (0,-0.5);
		
		\draw[->, > = latex', CornflowerBlue, path fading=north, decorate, decoration={snake, amplitude=1mm, segment length=5mm}] (0,4.25) -- (0,4.75);

		\draw[->, > = latex', PastelOrange, decorate, decoration={snake, amplitude=0.5mm, segment length=10mm}] (-1,0) to(1,0);
		\draw[->, > = latex', PastelOrange, decorate, decoration={snake, amplitude=0.5mm, segment length=10mm}] (-1,2.25) to (0,2.25);
		
		\draw[PastelOrange] (-2,0) node {$P_{n-1}$};
		\draw[PastelOrange] (-2,2.25) node {$P_{n}$};
		
		\draw[] (-2,0.75) node {$H_{\alpha(n,N)}$};
		\draw[] (-2,3) node {$H_{\alpha(n+1,N)}$};
		
		
		\filldraw[LavenderMagenta] (6,1.875) circle (2pt);
		\filldraw[LavenderMagenta] (6,2.625) circle (2pt);
		\filldraw[LavenderMagenta] (6,1.125) circle (2pt);
		
		\draw[->, > = latex', CornflowerBlue, decorate, decoration={snake, amplitude=1mm, segment length=5mm}] (6,2) -- (6,2.5);
		\draw[->, > = latex', CornflowerBlue, decorate, decoration={snake, amplitude=1mm, segment length=5mm}] (6,1.25) -- (6,1.75);
		\draw[->, > = latex', CornflowerBlue, path fading=south, decorate, decoration={snake, amplitude=1mm, segment length=5mm}] (6,0.5) -- (6,1);
		\draw[->, > = latex', CornflowerBlue, path fading=north, decorate, decoration={snake, amplitude=1mm, segment length=5mm}] (6,2.75) -- (6,3.25);
		
		\draw[->, > = latex', PastelOrange, decorate, decoration={snake, amplitude=0.5mm, segment length=8mm}] (5.25,1.125) to (5.875,1.125);
		\draw[->, > = latex', PastelOrange, decorate, decoration={snake, amplitude=0.5mm, segment length=8mm}] (6.125,1.125) to (6.75,1.125);
		\draw[->, > = latex', PastelOrange, decorate, decoration={snake, amplitude=0.5mm, segment length=8mm}] (5.25,1.875) to (5.875,1.875);
		
		\draw[PastelOrange] (4.5,1.125) node {$P_{n-1}$};
		\draw[PastelOrange] (4.5,1.875) node {$P_{n}$};
		
		\draw [line width=0.3mm, -{Stealth[length=4mm, round]}] (2,1.875) -- (3,1.875);
			\end{tikzpicture}
		\caption{The construction of the strong minor $A$ of $D$ in the proof of~\cref{thm:necklace}.
		The magenta outlined areas form components of $\bigcup_{\ell \in \N \setminus \alpha(N)} H_\ell^N$.}
		\label{fig:construction_aux_graph}
		\end{figure}
	\end{center}
	
	We apply \cref{thm:main} to $A$ and $\R$ to obtain a subdivision $A'$ of either a bidirected quarter-grid or a cyclic quarter-grid such that each vertical out-ray of $A'$ is an element of $\R$.
	The family $\R'$ of vertical out-rays of $A'$ corresponds to a family $\cN' \subseteq \cN$ of necklaces.
	Furthermore the (disjoint) horizontal $\R'$-paths in $A'$ correspond to disjoint $\cN'$-paths in $D'$, where the endpoints of the $\cN'$-paths appear in the same pattern on $\cN'$ as on $\R'$.
	
	Finally, we consider for each $N \in \cN$ the witness $(\hat{H}_k^N)_{k \in \N}$, where $\hat{H}_k^N$ is a strong component of $\bigcup_{\ell \in \N \setminus \alpha(N)} H_\ell^N$ if $k$ odd and $\hat{H}_k^N$ is $\hat{H}_\ell^N$ for some $\ell \in \alpha(N)$ if $k$ even.
	Note that the disjoint $\cN'$--paths corresponding to paths in $A'$ do not start or end in the interior of $\hat{H}_k^N$ for $k$ even.
	Thus the union of $\bigcup \cN'$ and the $\cN'$-paths forms either a bidirected necklace grid or a cyclic necklace grid whose vertical necklaces are elements of $\cN$.
\end{proof}

	In the remaining part of this paper we show that both types of necklace grids are necessary in a statement like~\cref{thm:necklace} even if we the relax the constraints of~\cref{thm:necklace} so that the out-rays are equivalent to $(R_i)_{i \in I}$.

\begin{lemma}\label{lem:no_bidirected_necklace_grid}
	There exists a cyclic necklace grid that does not contain a bidirected necklace grid.
\end{lemma}

\begin{proof}
	Let $D$ be an ascending cyclic necklace grid whose vertical necklaces $R_1, R_2, \dots$ are isomorphic to $\mathcal{D}(S)$ for an undirected ray $S$ and whose arches are single edges.
	Further, let $\mathcal{A}$ be the set of arches of $D$.
	We suppose for a contradiction that $D$ contains a bidirected necklace grid $D'$ and let $N_1, N_2, \dots$ be the vertical necklaces of~$D'$.
	
	\begin{claim}\label{clm:necklace_infinitely_often}
		There exists $i \in \N$ such that $N_i$ contains infinitely many arches of $D$.
	\end{claim}

	\begin{claimproof}
		Suppose for a contradiction that each necklace $N_i$ contains only finitely many arches of $D$.
		Then there exists a tail $N_i'$ of $N_i$ for every $i \in \N$ such that $N_i'$ does not contain an arch.
		As $R_1, R_2, \dots$ are the strong components of $D - \mathcal{A}$, $N_i'$ is contained in some necklace $R_{f(i)}$.
		Since $R_{f(i)}$ is isomorphic to $\mathcal{D}(S)$, $N_i'$ is a tail of $R_{f(i)}$ and in particular, $f(i) \neq f(j)$ for every $i \neq j \in \N$.
		
		Thus there exist $\ell < m < n$ such that $f(\ell) < f(m) < f(n)$.
		Since $D'$ is a bidirected necklace grid, there are infinitely many disjoint $N_n'$--$N_m'$~paths that avoid $N_\ell'$.
		Thus there are infinitely many disjoint $R_{f(n)}$--$R_{f(m)}$~paths in $D$ that avoid $N_\ell'$.
		This contradicts the facts that every $R_{f(n)}$--$R_{f(m)}$~path in $D$ hits $R_{f(\ell)}$ and $R_{f(\ell)} \setminus N_\ell'$ is finite.
	\end{claimproof}

	Let $i \in \N$ be as in~\Cref{clm:necklace_infinitely_often} and let $j \in \N \setminus \{i\}$ be arbitrary.
	Further, let $x$ be some vertex of $N_j$ and let $G$ be a girder of $D$ with the property $x \in \Down{G}_\N$, where the set $\Down{G}_\N \subseteq V(D)$ is defined as expected.
	
	We consider some tail $N_i'$ of $N_i$ that avoids the finite set $\Down{G}_\N$ and let $A$ be some arch in $N_i'$, which exists by~\Cref{clm:necklace_infinitely_often}.
	Let $a$ be the start vertex and $b$ be the end vertex of $A$.
	Since $N_i'$ is strongly connected, there exists a $b$--$a$~path in $N_i'$.
	In particular, there exists an $R_1$--$\{a\}$~path $P$ in $N_i'$ that does not contain an arch.
	
	\begin{claim}\label{clm:finite_component}
		The vertex $x$ is contained in a finite strong component of $D - P$.
	\end{claim}
	\begin{claimproof}
		Note that $D - \mathcal{A}$ is a planar digraph.
		We consider a planar drawing of $D - \mathcal{A}$ similar to~\cref{fig:dominated_directed_quarter_grid_out}.
		The path $P$ starts in this drawing on the left hand side, ends on the right hand side and avoids the set $\Down{G}_\N$ since $P \subseteq N_i'$.
		This implies that $V(P)$ separates $\Down{G}_\N$ from all but finitely many vertices of $D - \mathcal{A}$.
		In particular, the vertex $x \in \Down{G}_\N$ is contained in a finite weak component $X$ of $D - \mathcal{A} - P$.
		
		We show that $X$ is a strong component of $D - P$.
		Then $x$ is contained in the finite strong component $X$ of $D - P$.
		Let $s$ be the start vertex of $P$.
		
		If $s \in \Down{A}_\N$, then $X \cap V(R_1) \subseteq \Down{A}_\N$.
		Thus every arch ending in $X$ has its start vertex in $\Down{A}_\N$.
		Note that all start vertices of arches in $\Down{A}_\N$ are contained in $X \cup V(P)$.
		This implies that there is no arch starting in $V(D) \setminus (X \cup V(P))$ and ending in $ X$.
		If $s \notin \Down{A}_\N$, then $V(R_1) \setminus (X \cup V(P)) \subseteq V(R_1) \setminus \Down{A}_\N$.
		Thus every arch ending in $V(D) \setminus (X \cup V(P))$ avoids $\Down{A}_\N$.
		Note that all start vertices of arches in $X$ are contained in $\Down{A}_\N$.
		This implies that no arch starts in $X$ and ends in $V(D) \setminus (X \cup V(P))$.
		In both cases, $X$ is a strong component of $D - P$.
	\end{claimproof}

	By~\Cref{clm:finite_component}, $N_j$ has to contain a vertex of $P$.
	This contradicts the fact that $N_i, N_j$ are disjoint since $P \subseteq N_i$, and completes the proof.
\end{proof}

\begin{lemma}
	There exists a bidirected necklace grid that does not contain a cyclic necklace grid.
\end{lemma}

\begin{proof}
	Let $D$ be a bidirected necklace grid whose vertical necklaces are isomorphic to $\mathcal{D}(S)$ for an undirected ray $S$.
	Suppose for a contradiction that $D$ contains a cyclic necklace grid $D'$.
	Let $N_1, N_2$ and $N_3$ be distinct vertical necklaces of $D'$.
	We consider the underlying undirected graph $G(D')$ of $D'$ and let $R_i$ be an undirected ray contained in $N_i$ for every $i \in [3]$.
	Note that there are infinitely many disjoint undirected $R_i$--$R_j$~paths that avoid $R_k$ in $G(D')$ for every $i \neq j \in [3]$ with $ k \in [3] \setminus \{i,j\}$ since $D'$ is a cyclic necklace grid.
	
	Now we consider the rays $R_1, R_2, R_3$ in the underlying undirected graph $G(D)$ of $D$.
	Let $G$ be a girder of $D$ such that $\Down{G}$ contains vertices of $R_1, R_2$ and $R_3$.
	Further let $R_i'$ be the unique tail of $R_i$ that intersects $\Down{G}$ precisely in its first vertex for every $i \in [3]$.
	Note that $R_1', R_2'$ and $R_3'$ start in $V(G)$.
	We assume without loss of generality that the path $G$ hits the start vertices of $R_1', R_2'$ and $R_3'$ in this order.
	Then there exists no $R_1'$--$R_3'$~path that avoids $\Down{G}$ and $R_2'$, a contradiction.
	This completes the proof.
\end{proof}

\section*{Acknowledgement}
The author gratefully acknowledges support by a doctoral scholarship of the Studienstiftung des deutschen Volkes.

\bibliography{ref.bib}

\begin{bibdiv}
\begin{biblist}

\bib{amiri2016erdos}{article}{
      author={Amiri, Saeed~Akhoondian},
      author={Kawarabayashi, Ken-ichi},
      author={Kreutzer, Stephan},
      author={Wollan, Paul},
       title={The {E}rd{\H{o}}s-{P}{\'o}sa {P}roperty for directed graphs},
        date={2016},
     journal={arXiv preprint arXiv:1603.02504},
}

\bib{bowler2024connectoids}{article}{
      author={Bowler, Nathan},
      author={Reich, Florian},
       title={Connectoids {I}: a universal end space theory},
        date={2024},
     journal={arXiv preprint arXiv:2405.14704},
}

\bib{burger2020ends}{article}{
      author={B{\"u}rger, Carl},
      author={Melcher, Ruben},
       title={Ends of digraphs {I}: basic theory},
        date={2020},
     journal={arXiv preprint arXiv:2009.03295},
}

\bib{diestel}{book}{
      author={Diestel, Reinhard},
       title={Graph {T}heory},
     edition={6},
   publisher={Springer (print edition); Reinhard Diestel (eBooks)},
        date={2024},
}

\bib{georgakopoulos2024full}{article}{
      author={Georgakopoulos, Agelos},
      author={Hamann, Matthias},
       title={A full {H}alin grid theorem},
        date={2024},
     journal={Discrete \& Computational Geometry},
       pages={1\ndash 13},
}

\bib{halin1965maximalzahl}{article}{
      author={Halin, Rudolf},
       title={{\"U}ber die {M}aximalzahl fremder unendlicher {W}ege in {G}raphen},
        date={1965},
     journal={Mathematische Nachrichten},
      volume={30},
      number={1-2},
       pages={63\ndash 85},
}

\bib{hamann}{article}{
      author={Hamann, Matthias},
      author={Heuer, Karl},
       title={Infinite grids in digraphs},
        date={2024},
     journal={arXiv preprint arXiv:2412.03302},
}

\bib{heuer2017excluding}{inproceedings}{
      author={Heuer, Karl},
       title={Excluding a full grid minor},
organization={Springer},
        date={2017},
   booktitle={Abhandlungen aus dem {M}athematischen {S}eminar der {U}niversit{\"a}t {H}amburg},
      volume={87},
       pages={265\ndash 274},
}

\bib{kurkofka2022strengthening}{article}{
      author={Kurkofka, Jan},
      author={Melcher, Ruben},
      author={Pitz, Max},
       title={A strengthening of {H}alin's grid theorem},
        date={2022},
     journal={Mathematika},
      volume={68},
      number={4},
       pages={1009\ndash 1013},
}

\bib{infinite}{article}{
      author={Reich, Florian},
       title={A star-comb lemma for infinite digraphs},
        date={2024},
     journal={arXiv preprint arXiv:2406.04877},
}

\bib{zuther1998ends}{article}{
      author={Zuther, J{\"o}rg},
       title={Ends in digraphs},
        date={1998},
     journal={Discrete mathematics},
      volume={184},
      number={1-3},
       pages={225\ndash 244},
}

\end{biblist}
\end{bibdiv}

\end{document}